\documentclass[10pt]{article}
\usepackage[scr=boondox]{mathalfa}
\usepackage{amsmath, amsfonts, amsthm, amssymb, hyperref, setspace, geometry}
\title{\textbf{On the finite solvable PNC-groups}\thanks{\footnotesize  Corresponding
  author. \newline\hspace*{0.45cm} \scriptsize\emph{E-mail addresses:}
      zsmcau@cau.edu.cn\,(S. Zhang); zhencai688@sina.com\,(Z. Shen).}}
 
\author{Shengmin Zhang, Zhencai Shen\\
\quad
\\
{\small College of Science,
China Agricultural University,
Beijing 100083, China}}
\date{}
\newtheorem{theorem}{Theorem}[section]
\newtheorem{proposition}[theorem]{Proposition}
\newtheorem{lemma}[theorem]{Lemma}

\theoremstyle{definition}
\newtheorem{definition}[theorem]{Definition}

\newtheorem{question}[theorem]{Question}
\newtheorem{remark}[theorem]{Remark}

\newtheorem{observation}[theorem]{Observation}
\allowdisplaybreaks

\expandafter\let\expandafter\oldproof\csname\string\proof\endcsname
\let\oldendproof\endproof
\renewenvironment{proof}[1][\proofname]{%
  \oldproof[\bfseries\scshape #1]%
}{\oldendproof}

\usepackage{stackengine,scalerel}
\stackMath
\def\trianglelefteqslant{\ThisStyle{\mathrel{%
  \stackinset{r}{.75pt+.15\LMpt}{t}{.1\LMpt}{\rule{.3pt}{1.1\LMex+.2ex}}{\SavedStyle\leqslant}%
}}}
\renewcommand{\unlhd}{\trianglelefteqslant}

\renewcommand{\leq}{\leqslant}
\renewcommand{\geq}{\geqslant}
\begin{document}
\maketitle
\begin{abstract}
A subgroup $H$ of a finite group $G$ is said to be an NC-subgroup of $G$,   if $ H^G N_G (H) =G$, where $H^G$ denotes the normal closure of $H$ in $G$. A finite group $G$ is called a PNC-group, if any subgroup of $G$ is an NC-subgroup of $G$, and $G$ is said to be an ON-group, if for any subgroup $H$ of $G$, either $N_G (H)=H,\,H^G=G$, or $H \unlhd G$. In this paper, we firstly investigate the basic properties of solvable PNC-groups, and then give several sufficient conditions for $G$ to be a solvable PNC-group. In the end of this paper, we discover some characterizations for minimal non-ON-groups, ON-groups, non-abelian simple groups whose second maximal subgroups are solvable PNC-groups and   groups whose proper (maximal) subgroups are solvable PNC-groups.
\end{abstract}
\section{Introduction}
All groups considered in this paper will be finite. Let $G$ be a finite group. Then $G$ is said to be a PN-group, if all minimal subgroups of $G$ are normal in $G$. This concept was first investigated by Buckley in {{\cite{BU}}}. Inspired by the definition of PN-groups, In {{\cite{SL}}} , S. Li introduced a generalization of PN-groups, which is called PE-groups. In fact, the normality in the definition of PN-groups could be generalised, hence we have the following generalization:
\begin{definition}
A subgroup $H$ of a group $G$ is called an NE-subgroup of $G$, if $N_G (H) \cap H^G  = H$, where $H^G$ denotes the normal closure of $H$ in $G$. 
\end{definition}
As a natural way, S. Li gave a new definition in {{\cite{SL}}} as follows:
\begin{definition}
A group $G$ is said to be a PE-group, if any minimal subgroup of $G$ is an NE-subgroup of $G$.
\end{definition}
In that paper, S. Li classified the minimal non-PE groups and restricted the Fitting height by 3. Inspired by S. Li and his works, Y. Li studied the further relationship between finite solvable $\mathscr{T}$-groups and NE-subgroups, and obtained several new characterizations for supersolvablility and nilpotency in {{\cite{YL}}}. Recall that $G$ is said to be a $\mathscr{T}$-group, if any subnormal subgroup of $G$ is normal in $G$. A group $G$ is a $\mathscr{C}_p$-group, if every subgroup of a Sylow $p$-subgroup $P$ of $G$ is normal in $N_G (P)$, where $p$ is a prime divisor of $|G|$. A group is called a $\overline{\mathscr{T}}$-group, if all of its subgroups are $\mathscr{T}$-groups. A subgroup
$H$ of $G$ is said to be pronormal in $G$, if for each $g \in G$, $H$ and $H^g$ are already conjugate in $\langle H,H^g \rangle$. A subgroup $H$ of $G$ is called normally embedded in $G$, if every Sylow $p$-subgroup of $H$ is a Sylow $p$-subgroup of some normal subgroup of $G$,
for all $p \in \pi (G)$. Finally, as introduced in {{\cite{MA}}}, we call a subgroup $H$ of $G$ an $\mathscr{H}$-subgroup of $G$ if $N_G (H) \cap H^g \leq H$, for any $g \in G$. Combining the characterization of NE-subgroups, we conclude the following theorem for the characterizations of solvable $\mathscr{T}$-groups.
\begin{theorem}[{{\cite[Theorem 1.1]{YL}}} and {{\cite[Theorem 3.2]{YL}}}]\label{NE}
Let $G$ be a finite group, then the following are equivalent:
\begin{itemize}
\item[(1)] $G$ is a $\mathcal{T}$-group.
\item[(2)] $G = L \rtimes M$, where $L = [G,G^{'}]$ is an abelian normal Hall subgroup of $G$ and $M$ is a Dedekind group. 
\item[(3)] $G$ is a supersolvable $\mathcal{T}$-group.
\item[(4)] $G$ is a $\overline{T}$-group.
\item[(5)] $G$ satisfies $\mathscr{C}_{p}$ for all $p \in \pi (G)$.
\item[(6)] For all $p \in \pi (G)$, the $p$-subgroups of $G$ are pronormal in $G$.
\item[(7)] Every subgroup of $G$ is an $\mathscr{H}$-subgroup of $G$.
\item[(8)] For all $p \in \pi (G)$, the $p$-subgroups of $G$ are $\mathscr{H}$-subgroup of $G$.
\item[(9)] All subgroups of $G$ are normally embedded in $G$. 
\item[(10)] Every subgroup of $G$ is an NE-subgroup of $G$.
\item[(11)] For all $p \in \pi (G)$, the $p$-subgroups of $G$ are   NE-subgroups of $G$.
\end{itemize}
\end{theorem}
Similarly, inspired by the characterizations and the concept of NE-subgroups above, we introduce the concept of NC-subgroups and PNC-subgroups as follows.
\begin{definition}
A subgroup $H$ is said to be an NC-subgroup of $G$, if $ H^G N_G (H) =G$, where $H^G$ denotes the normal closure of $H$ in $G$.
\end{definition}
As a natural way of generalization, we introduce the following definition.
\begin{definition}
A group $G$ is called a PNC-group, if any subgroup of $G$ is an NC-subgroup of $G$. 
\end{definition}
A natural question is: can we give solvable $\mathscr{T}$-groups a new characterization which is related to NC-subgroups? Actually, a solvable $\mathscr{T}$-group may not necessarily be a solvable PNC-group. Thus PNC-groups are stronger than PE-groups and PN-groups, which showcases its potential to do some further research. Recall that a group $G$ is said to be an NSN-group, if
every subgroup $H$ of $G$ is either normal in $G$ or self-normalizing, that is, either $N_G (H)=H$ or $H \unlhd G$. In {{\cite{NSN}}}, Z. Han {\it et al.} classified the minimal non-NSN groups. Now we combine the property of PNC-groups and NSN-groups, and so give the concept of ON-groups as follows: 
\begin{definition}
A group $G$ is said to be an ON-group, if for any subgroup $H$ of $G$, either $N_G (H)=H,\,H^G=G$, or $H \unlhd G$.
\end{definition}
If a group $G$ is an ON-group, one can easily find that $G$ is a PNC-group, and $G$ is an NSN-group as well. Hence one of our aim in this paper is to classify the ON-groups.

In section \ref{10002}, we investigate the fundamental properties of solvable PNC-groups, and give some criteria for supersolvability, $p$-nilpotency related to PNC-groups. Restriction for the nilpotence class of the generalised Fitting subgroup of solvable PNC-groups are given as well. In section \ref{10003}, we study the inheritance of solvable PNC-groups, and several examples are given to support our conclusions. In section \ref{10004}, we show some sufficient conditions for a group $G$ to be a solvable PNC-groups. Some equivalent conditions for dihedral groups and generalised Quaternion groups are also given in this paper. In section \ref{10005}, we give some characterizations for non-PE-groups whose proper subgroups are solvable PNC-groups, non-PE-groups whose proper subgroups are solvable PNC-groups, non-PE-groups whose proper subgroups are ON-groups, and classify the ON-groups. Characterizations for groups whose maximal subgroups are solvable PNC-groups are also given in this paper. In particular, we classify the non-abelian simple groups whose second maximal subgroups are solvable PNC-groups. 
\section{Basic properties for solvable PNC-groups}\label{10002}
It is clear that simple group and abelian group are PNC-groups. A natural question is, if $G$ is nilpotent, supersolvable, etc, is $G$ still a PNC-group? What will $G$ be like if certain subgroups of $G$ is a solvable PNC-group? Is a solvable PNC-group $p$-nilpotent for some $p\,|\,|G|$? In this section, we will answer the questions above and give some specific criteria for solvable PNC-groups.
\begin{lemma}\label{T}
Assume that $G$ is a PNC-group, then $G$ is a $\mathcal{T}$-group.
\end{lemma}
\begin{proof}
Let $H$ be a subgroup of $G$ such that $H \unlhd \unlhd G$. Then there exists a subnormal series 
$$H = H_0 \unlhd H_1 \unlhd \cdots \unlhd H_n = G.$$
It is clear that $H_{n-1} \unlhd G$, and $N_G (H_{i-1}) \geq H_i$, $i=1,2,\cdots,n$. Assume that $H_i$ is normal in $G$, then $H_{i-1} ^G \leq H_i \leq N_G (H_{i-1})$. It follows that 
$$G = N_G (H_{i-1}) H_{i-1} ^G = N_G (H_{i-1}).$$
Thus $H_{i-1} \unlhd G$. Therefore we conclude that $H_i \unlhd G$, $i=0,1,\cdots,n$. In particular, $H = H_0 \unlhd G$. By the choice of $H$, $G$ is a $\mathcal{T}$ group.
\end{proof}
\begin{proposition}
Let $G$ be a nilpotent group. Then $G$ is a PNC-group if and only if $G$ is a Dedekind group.
\end{proposition}
\begin{proof}
Assume that $G$ is a PNC-group. Let $H$ be a subgroup of $G$. Since $G$ is nilpotent, $H \unlhd \unlhd G$. By Lemma \ref{T}, $H \unlhd G$. By the choice of $H$, we indicate that $G$ is a Dedkind group.

Assume that $G$ is a Dedekind group. Let $H$ be a subgroup of $G$. Since $H \unlhd G$, by the choice of $H$, $G$ is a PNC-group.
\end{proof}
\begin{proposition}\label{supersolvable}
Assume that a solvable group $G$ is a PNC-group, then $G$ is supersolvable.
\begin{proof}
Since $G$ is solvable, there exists a composition series
$$1=A_1 \unlhd A_2 \unlhd \cdots \unlhd A_{n-1} \unlhd A_{n} =G.$$
where each composition factor $A_i/A_{i-1}$ is cyclic of prime order. One can easily find that $A_{i} \unlhd\unlhd G$, $i=1,2,\cdots,n$. By Lemma \ref{T} we conclude that $A_i \unlhd G$, $i=1,2,\cdots,n$. Therefore $G$ is supersolvable. 
\end{proof}
\end{proposition}

\begin{remark}
We predicate that a supersolvable group $G$ is not necessarily a PNC-group. For example, let $G = C_2 ^2 \rtimes C_4$. It is clear that $G$ is supersolvable, and there exists a subgroup $C_4$ such that 
$$N_G (C_4)=C_2 \times C_4,\,C_4 ^G \leq C_2 \times C_4.$$
Thus the converse of proposition \ref{supersolvable} is not true.
\end{remark}
\begin{observation}
Let $G$ be a finite group and $K \leq G$. Then $K$ is an NC-subgroup of $G$ if and only if $[K,G] N_G (K) = G$.
\end{observation}
\begin{proof}
Clearly $[K,G] \unlhd G$. Therefore we conclude that $K[K,G]=K^G$. It follows that 
$$K^G N_G (K) = G \Leftrightarrow ([K,G]K)N_G (K) = G \Leftrightarrow [K,G] N_G (K) = G. \qedhere$$
\end{proof}
\begin{proposition}\label{NE}
Assume that $G$ is a solvable PNC-group, then the following holds:
\begin{itemize}
\item[(1)] $G$ is a $\mathcal{T}$-group.
\item[(2)] $G = L \rtimes M$, where $L = [G,G^{'}]$ is an abelian normal Hall subgroup of $G$ and $M$ is a Dedekind group. 
\item[(3)] $G$ is a supersolvable $\mathcal{T}$-group.
\item[(4)] $G$ is a $\overline{T}$-group.
\item[(5)] $G$ satisfies $\mathscr{C}_{p}$ for all $p \in \pi (G)$.
\item[(6)] For all $p \in \pi (G)$, the $p$-subgroups of $G$ are pronormal in $G$.
\item[(7)] Every subgroup of $G$ is an $\mathscr{H}$-subgroup of $G$.
\item[(8)] For all $p \in \pi (G)$, the $p$-subgroups of $G$ are $\mathscr{H}$-subgroup of $G$.
\item[(9)] All subgroups of $G$ are normally embedded in $G$. 
\item[(10)] Every subgroup of $G$ is an NE-subgroup of $G$.
\item[(11)] For all $p \in \pi (G)$, the $p$-subgroups of $G$ are   NE-subgroups of $G$.
\end{itemize}
\end{proposition}
\begin{proof}
By Lemma \ref{T} we have $G$ is a solvable $\mathcal{T}$- group. It follows from {{\cite[Theorem 1.1]{YL}}} that (1) to (9) hold. By {{\cite[Theorem 3.2]{YL}}}, (10) and (11) hold.
\end{proof}

\begin{observation}
Assume that $G$ is a solvable PNC-group, then the following holds:
\begin{itemize}
\item[(1)] Let $H$ be a subgroup of $G$, then $(N_G (H))^G = G$.
\item[(2)] Assume that there is no subgroup $H$ of $G$ such that $N_G (H) <G$ and $N_G (H) = G$, then $G$ is a Dedekind group.
\end{itemize}
\end{observation}
\begin{proof}

(1) By proposition \ref{NE} (7), $H$ and $N_G (H)$ are $\mathscr{H}$-groups. It follows from {{\cite[Theorem 6(1)]{MA}}} that $N_G (H) = N_G (N_G (H))$. Hence we conclude that 
$$G = N_G (H) ^G N_G (N_G (H)) =  N_G (H) ^G.$$

(2) It follows directly from (1).
\end{proof}

\begin{lemma}\label{nilpotent}
Let $G$ be a solvable PNC-group. Then every nilpotent subgroup of $G$ is a Dedekind group.
\end{lemma}
\begin{proof}
Let $H$ be a nilpotent subgroup of $G$, and let $K$ be a subgroup of $H$. One can easily find that $K \unlhd \unlhd H$.  By proposition \ref{NE} (7) and {{\cite[Lemma 2.1(2)]{YL}}}, we have $K \unlhd H$. By the choice of $K$,  we indicate that $H$ is a Dedekind group.
\end{proof}

\begin{proposition}
Let $G$ be a finite group. Then the following holds:
\begin{itemize}
\item[(1)] Assume that every proper non-nilpotent subgroup $H$ is a solvable PNC-group, then $G$ is solvable.
\item[(2)] Let $G$ be a PNC-group, and $p = \max(\pi (G))$. Suppose that there exists a subgroup $P \leq G$ such that $|P| =p$, then $P$ is a normal subgroup of $G$.
\end{itemize}
\end{proposition}
\begin{proof}
(1) For any proper non-nilpotent subgroup $H$ of $G$, it follows from proposition \ref{NE} (10) that every subgroup of $H$ is an NE-group of $H$. Hence $H$ is a PE-group. By {{\cite[Theorem 3.4]{LG}}}, $G$ is solvable.

(2) Since $G$ is a PNC-group, $G$ is a PE-group. Hence by {{\cite[Theorem 3.4]{LG}}}, we have $P \unlhd G$.
\end{proof}
\begin{proposition}\label{normal complement}
Let $G$ be a solvable PNC-group, and $p = \min \lbrace \pi (G) \rbrace$. Then $G$ is $p$-nilpotent.
\end{proposition}
\begin{proof}
One can easily conclude from proposition \ref{NE} (10) that every subgroup of $G$ is an NE-subgroup of $G$. If $p=2$, For $S \in {\rm{Syl}}_2 (G)$, and $P \leq S$ where $P$ is a cyclic subgroup with the order of 2 or 4, it follows from Lemma \ref{nilpotent} that $S$ is a Dedekind group and $P$ is a normal subgroup of $S$. Hence $P$ is a subnormal subgroup of $N_G (S)$. By {{\cite[Lemma 2.1(2)]{YL}}}, $P \unlhd N_G (S)$. By the choice of $P$ we indicate from {{\cite[Lemma 2.5(1)]{YL}}} that $G$ is 2-nilpotent.

If $p>2$, then $G$ is a solvable PNC-group with odd order. Since a Dedekind group of odd order is an abelian group, it follows from Lemma \ref{nilpotent} that every Sylow subgroup of $G$ is abelian. As every subgroup of $G$ with order $p$ is an NE-subgroup of $G$, we conclude from {{\cite[Lemma 5]{SL}}} that $G$ is $p$-nilpotent.
\end{proof}
\begin{observation}
Let $G$ be a solvable PNC-group. Then for any $p \in \pi (G)$ and $p$-subgroup $P$, $P \in {\rm{Syl_p}} (P^G)$ and $|G|_p = |N_G (P)|_p$.
\end{observation}
\begin{proof}
Let $P \leq S \in {\rm{Syl_p}} (G)$. Then by Lemma \ref{nilpotent} we conclude that $P \unlhd S$. Hence $N_G (P) \geq S$ and $|G|_p = |N_G (P)|_p$. It follows from Dedekind identity that 
$$|G|_p = \frac{|N_G (P)|_p \cdot |P^G|_p}{|P|_p} = \frac{|G|_p \cdot |P^G|_p}{|P|_p} $$
Therefore $|P^G|_p = |P|_p$, which implies that $P \in {\rm{Syl_p}} (G)$.
\end{proof}
\begin{proposition}
Let $G$ be a group. If every non-nilpotent proper subgroup of
$G$ is a solvable PNC-group, then $G$ is solvable.
\end{proposition}
\begin{proof}
It follows from proposition \ref{NE} that any non-nilpotent proper subgroup of $G$ is a PE-group. Hence by {{\cite[Theorem 3.4]{LG}}} we conclude that $G$ is solvable.
\end{proof}
\begin{lemma} \label{component}
Let $G$ be a finite group. 
\begin{itemize}
\item[(1)] Suppose that $G =AB$, where $A,B \unlhd G$. If for any $a \in A$ and $b \in B$, we have $ab=ba$, then $G' \leq A' B'$.
\item[(2)] Suppose that $G = A_1 A_2 \cdots A_n$, where $A_1,A_2,\cdots,A_n \unlhd G$. If for any $a_i \in A_i$, $a_j \in A_j$ with $i \neq j$, $a_i a_j = a_j a_i$, then $G' \leq A_1 ' A_2 ' \cdots A_n '$.
\end{itemize}
\end{lemma}
\begin{proof}
(1) Let $a_1,a_2 \in A$ and $b_1,b_2 \in B$. It follows that 
\begin{align*}
[a_1 b_1,a_2 b_2] &= [a_1 b_1, b_2] [a_1 b_1, a_2]^{b_2} = [a_1 b_1, b_2] [a_1 b_1, a_2] \\
&= [a_1,b_2]^{b_1} [b_1,b_2] [a_1,a_2]^{b_1} [b_1,a_2] = [b_1,b_2][a_1,a_2] \in A'B'.
\end{align*}
Hence we conclude that $G' \leq A'B'$.

(2) It follows directly from (1) by induction on $n$.
\end{proof}
\begin{proposition}\label{Fitting Group}
Let $G$ be a PNC-group. Then the nilpotence class of $F^{*} (G)$ is at most 2.
\end{proposition}
\begin{proof}
By Lemma \ref{T} we conclude that $G$ is a $\mathcal{T}$-group. Hence any component of $G$ is normal in $G$. By the definition of $E(G)$, there exists some components $E_1,E_2,\cdots,E_n$ with $[E_i,E_j]=1,\,i,j=1,2,\cdots,n$ such that
$$E(G) = E_1 \cdot E_2 \cdots E_n. $$
By lemma \ref{component}, it follows that
$$E(G)' \leq E_1 ' \cdot E_2 ' \cdots E_n ' \leq Z(E_1) \cdot Z(E_2) \cdots Z(E_n) \leq Z(E(G)). $$
Since $F(G)$ is nilpotent, we indicate from Lemma \ref{nilpotent} that $F(G)$ is a Dedekind group. Therefore we conclude that $F(G)' \leq Z(F(G))$. Since $F^{*} (G) = E(G) F(G)$ and $[E(G),F(G)]=1$, by lemma \ref{component} again we have 
$$F^{*} (G)' \leq E(G)' F(G)' \leq Z(E(G)) \cdot Z(F(G)) \leq Z(F^{*} (G)).$$
Thus the nilpotence class of $F^{*} (G)$ is at most 2.
\end{proof}
We end this section with some basic observations about PNC-group.
\begin{proposition}
Let $G$ be a solvable PNC-group, then the following hold:
\begin{itemize}
\item[(1)] The Fitting height of $G$ is bounded by 3.
\item[(2)] There exists a PE-group $K$ such that $G = K F(G)$, where $F(G)$ denotes the Fitting subgroup of $G$.
\item[(3)] Let $G$ be a finite group of odd order. Then we have 
\item[(4)] For any $p \in \pi (G)$, the $p$-length of $G$ is bounded by 1.
\item[(5)] For any $p \in \pi (G)$, $G$ is $p$-nilpotent or $p$-perfect.
\item[(6)] The Frattini subgroup of $G$ is abelian.
\item[(7)] $G$ is metabelian.
\item[(8)] If $G$ is of odd order, then $G' \leq Fit (G)$.
\end{itemize}
\end{proposition}
\begin{proof}
Since every subgroup of $G$ is an NE-subgroup, $G$ is a PE-group. It follows directly from {{\cite[Theorem 1]{SL}}} that (1) and (2) hold. 

(3) Since any Sylow subgroup of $G$ has odd order, it follows from Lemma \ref{nilpotent} that any Sylow subgroup of $G$ is abelian. We indicate from {{\cite[Exercise 7.2.3]{KS}}} that $G' \cap Z(G)=1$.

(4) It follows from proposition \ref{NE} (5) that $G$ satisfies $\mathscr{C}_{p}$ for all $p \in \pi (G)$. Since $G$ is supersolvable, $G$ is $p$-solvable for any $p \in \pi (G)$. Thus by {{\cite[Corollary of Theorem 3]{DR2}}} we conclude that for any $p \in \pi (G)$, the $p$-length of $G$ is bounded by 1.

(5) It follows directly from {{\cite[Theorem 3]{DR3}}}.

(6) It follows directly from {{\cite[Exercise 13.4.5]{DR2}}}.

(7) By proposition \ref{NE} (3), $G$ is a supersolvable $\mathcal{T}$-group. It follows from {{\cite[Theorem 2.3.1]{DR}}} that $G$ is metabelian.

(8) It follows from Lemma \ref{nilpotent} that all Sylow subgroups of $G$ are Dedekind groups. Since $G$ is of odd order, we conclude that all Sylow subgroups of $G$ are abelian groups. Hence $G$ is a SSA-group. Since $G$ is supersolvable, by {{\cite[Theorem 19]{PM}}} we conclude that $G' \leq Fit (G)$.
\end{proof}

\begin{remark}
(1) It is natural for us to compare NC-subgroup with NE-subgroup. However, compared to NE-subgroup, NC-subgroup fails to satisfy the basic properties which NE-subgroup possesses such as  {{\cite[Lemma 2.1]{YL}}}. For example, let $G ={\rm U}_2({\mathbb F}_3) $. One can find that there exists a subgroup $C_2 \times C_4$ such that $N_G (C_2 \times C_4) = C_4 ^2 < C_4\wr C_2$, and $(C_2 \times C_4) ^G = G = {\rm U}_2({\mathbb F}_3)$. Hence we have  $N_G (C_2 \times C_4) \cdot  (C_2 \times C_4) ^G = G$. Now consider a subgroup $ C_2 \times C_4$ where $ C_2 \times C_4 \leq C_4\wr C_2 $, and $C_2 \times C_4 \ntrianglelefteq C_4\wr C_2$. It follows that $N_{C_4\wr C_2} (C_2 \times C_4) = C_4\circ D_4 \unlhd C_4\wr C_2$, and $C_4\circ D_4 < C_4\wr C_2$. Therefore we conclude that $(C_2 \times C_4)^{C_4\wr C_2} \cdot N_{C_4\wr C_2} (C_2 \times C_4) = C_4\circ D_4 < C_4\wr C_2$. Thus $C_2 \times C_4$ is an NC-subgroup of $G$, but it is not a NC-subgroup of $C_4\wr C_2$, which implies that NC-subgroup fails to satisfy {{\cite[Lemma 2.1(1)]{YL}}}.

(2) Now let $G = {\rm U}_2({\mathbb F}_3)$ again. One can find that there exists a NC-subgroup $C_2 \times C_4$ such that $N_G (C_2 \times C_4) = C_4 ^2 < C_4\wr C_2$. Since we have $N_G (C_4 ^2) = C_4\wr C_2$, it follows that $C_2 \times C_4 \unlhd \unlhd G$, and  
$C_2 \times C_4 \ntrianglelefteq C_4\wr C_2$. Thus we indicate that NC-subgroup fails to satisfy {{\cite[Lemma 2.1(2)]{YL}}}.

(3) We predicate that the condition of minimal prime $p$ in  proposition \ref{normal complement} can not be optimized. In other words, for a solvable PNC-group $G$, if $p$ is not the minimal number in $\pi (G)$, then $G$ is not necessarily $p$-nilpotent. For example, let $G = C_5\times S_3$. It is clear that $G$ is a solvable PNC-group. However, the subgroup $C_{10}$ is not normal in $G$. Hence we conclude that $G$ is not 3-nilpotent.
\end{remark}
\section{Inheritance of PNC-groups}\label{10003}
It is well known that the inheritance of PNC-groups plays an important role in the structure of PNC-groups. We can compare these properties with PN-groups and PE-groups, and see if there are some properties that they have in common.
\begin{proposition}\label{1}
Let $H,K$ be PNC-groups, $(|H|,|K|)=1$, then $H \times K$ is a PNC-group.
\end{proposition}
\begin{proof}
Let $T$ be a subgroup of $G := H \times K$. It follows directly that 
$$\left(|T|/|T \cap H|,|T \cap H| \right) =1.$$
Schur-Zassenhaus Theorem shows that there exists $S \leq T$ such that $S$ is a complement of $H \cap T$ in $T$. We predicate that $(|H|,|T|)=|H \cap T|$. Otherwise $(|T|/|T \cap H|,|H|) \neq 1$, i.e. $(|S|,|H|) \neq 1$. Assume that $p \mid (|S|,|H|)$, then there exists $t \in S $ such that $o(t)=p$. Let $t =kh$, it follows that $t^p = k^p h^p =1$. Hence $1=k^p =h^p$. Since $(|H|,|K|)=1$, we have $k=1$, therefore $1 \neq t \in H$, a contradiction. Using the same method we conclude that $(|K|,|T|)=|K \cap T|$. As $|T| \mid |K||H|$, we indicate that
$$|K \cap T| \cdot |T \cap H| = |T|,~\text{i.e.}~\left( K \cap T \right)  \times \left( T \cap H \right) =T.$$
Hence we have $N_H (T \cap H) N_K (T \cap K) \leq N_G (T)$. Since $T^G \geq (T \cap H)^H \cdot (T \cap K)^K$, we conclude that 
\begin{align*}
T^GN_G (T)  &\geq  (T \cap H)^H \cdot (T \cap K)^K N_H (T \cap H) N_K (T \cap K)\\
& = \left( (T \cap H)^H N_H (T \cap H) \right) \left( (T \cap K)^K N_K (T \cap K) \right)\\
& =H \cdot K =G.
\end{align*}
Thus $T$ is an NC-subgroup of $G$. By the choice of $T$, $G$ is a PNC-group. 
\end{proof}

\begin{remark}
(1) Proposition \ref{1} is not necessarily true, if $(|H|,|K|) \neq 1$. Take $H = C_3,K =S_3$ for instance. Since there exists a subgroup $C_3$, such that $N_G (C_3) =C_3 \times C_3,\,C_3 ^G \leq C_3 \times C_3$. Thus proposition \ref{1} fails in this case. 

(2) Proposition \ref{1} is still not necessarily true, if we change $H \times K$ into $ H \ltimes K$. For example, let $H =C_3,K= D_4$. Then there exists a subgroup $S_3$, such that $N_G (S_3) =D_6, S_3 ^G \leq D_6$. Thus proposition \ref{1} fails in this case.

\end{remark}

\begin{proposition}\label{factor group}
Let $G$ be a PNC-group and $N \unlhd G$. Then $G/N$ is a PNC-group.
\end{proposition}
\begin{proof}
For $T/N \leq G/N$, where $N \leq T$ is a subgroup of $G$. It follows directly that 
\begin{align*}
gN \in N_{G/N} (T/N) & \Leftrightarrow (T/N)^{gN} =T/N \Leftrightarrow T^g /N =T /N \Leftrightarrow T^g =T \Leftrightarrow g \in N_G (T).
\end{align*}
Thus we have $N_G (T) /N = N_{G/N} (T/N)$. Also, one can easily find that
\begin{align*}
(T/N)^{G/N} =\langle g^{-1} t gN  \mid g \in G,\,t \in T \rangle = T^G /N .
\end{align*}
Thus we indicate that $  T^G /N \cdot  N_G (T) /N  = G/N$, hence we get 
$$N_{G/N} (T/N) \cdot (T/N)^{G/N} = N_G (T) /N \cdot T^G /N =G/N. $$ 
By the choice of $T$, we conclude that $G/N$ is a PNC-group.
\end{proof}
\begin{remark}
For a group $G$, if for any normal subgroup $N>1$ of $G$, $G/N$ is a PNC-group, then $G$ is not necessarily a PNC-group. Let $G = D_4$. It is well known that $D_4$ is not a PNC-group while all of its non-trivial quotient groups are PNC-groups. Thus the proposition fails in this case.
\end{remark}
\begin{proposition}
Let $G$ be a PNC-group, and $M \unlhd G$. Then $M$ is a PNC-group.
\end{proposition}
\begin{proof}
Let $U$ be a subgroup of $M$. Since $U^G \leq M$, $U^{U^G} \leq U^M$ and $G =U^G N_G (U)$, it follows from Dedekind modular law that
$$M = M \cap G = (U^G N_G (U)) \cap M = U^G (M \cap N_G (U))=  U^{N_G (U) U^G} N_{M} (U)= U^{U^G} N_{M} (U) = U^{M} N_{M} (U).$$
By the choice of $U$, we indicate that $M$ is a PNC-group.
\end{proof}
\begin{remark}
If $G$ is a PNC-group, then its subgroups are not necessarily PNC-groups. Take $G=C_7 \times A_5$ for example, then there exists a subgroup $A_4$ which is not a PNC-group. Hence the argument fails in this case.
\end{remark}
\begin{theorem}\label{equivalent 2}
Let $G$ be a group. Suppose that there exists a normal subgroup $\langle x \rangle \leq Z(G)$ with $|\langle x \rangle|=p$, where $p$ is a prime, and $|G|_p = p$. Then $G$ is a PNC-group if and only if $G / \langle x \rangle$ is a PNC-group. 
\end{theorem}
\begin{proof}
For necessity, we conclude from proposition \ref{factor group} that it is clear. For sufficiency, let $K$ be a subgroup of $G$. If $\langle x \rangle \leq K$, it follows from the same method in proposition \ref{factor group} that $N_G (K) / \langle x \rangle = N_{G/\langle x \rangle} (K/\langle x \rangle)$, $(K/\langle x \rangle)^{G/\langle x \rangle} = K^G / \langle x \rangle$. Since $G/N$ is a PNC-group, we conclude that $N_{G/\langle x \rangle} (K/\langle x \rangle) \cdot  (K/\langle x \rangle)^{G/\langle x \rangle} = G/ \langle x \rangle$. 
Hence we indicate that 
$$G/ \langle x \rangle = \left( K^G / \langle x \rangle \right) \cdot  \left( N_G (K) / \langle x \rangle  \right).$$
As $N_G (K),K^G \geq  \langle x \rangle$, we have $N_G (K) K^G =G$.

If $\langle x \rangle \not\leq K$, it follows that 
$$g \langle x \rangle \in N_{G/\langle x \rangle } (K \langle x \rangle /\langle x \rangle) \Leftrightarrow (K \langle x \rangle /\langle x \rangle)^{g \langle x \rangle } = K \langle x \rangle /\langle x \rangle \Leftrightarrow  K^g \langle x \rangle =K \langle x \rangle.$$
If $K^g \langle x \rangle =K \langle x \rangle$, we prove that $g \in N_G (K)$. By $\langle x \rangle \not\leq K$ we conclude that $K \cap \langle x \rangle =1$. Hence one can easily get that $(|K|,p)=1$. Thus  for $k \in K$, there exists $t \in \mathbb{N}$ such that $ tp \equiv 1 \,(\!\!\!\mod o(k))$. It follows from $K^g \langle x \rangle =K \langle x \rangle$ that $g^{-1} k^t g = k_1 x^j$, where $j$ is an integer and $k_1 \in K$. Therefore we indicate from $\langle x \rangle \leq Z(G)$ that 
$$(g^{-1} k^t g)^p = g^{-1} k g = k_1 ^p x^{jp} =k_1 ^p \in K.$$
Hence $g \in N_G (K)$ and we conclude that 
$$g \langle x \rangle \in N_{G/\langle x \rangle } (K \langle x \rangle /\langle x \rangle) \Leftrightarrow K^g \langle x \rangle =K \langle x \rangle \Leftrightarrow g \in N_G (K).$$
Thus we have $N_{G/\langle x \rangle } (K \langle x \rangle /\langle x \rangle) = N_G (K)/ \langle x \rangle$. It follows from $N_G (K \langle x \rangle) /\langle x \rangle$ that $N_G (K \langle x \rangle)=  N_G (K)$. Hence we indicate that 
\begin{align*}
G=N_G (K \langle x \rangle) \cdot  (K \langle x \rangle)^G = N_G (K \langle x \rangle) \cdot  \langle x \rangle K^G =\left(N_G (K \langle x \rangle)  \langle x \rangle \right) K^G =N_G (K) K^G
\end{align*}
Hence $K$ is an NC-subgroup of $G$. By the choice of $K$ we conclude that $G$ is a PNC-group.
\end{proof}
\begin{remark}
(1) If we change the condition $\langle x \rangle \leq Z(G)$ into $(p-1, |G|/p)=1$, then the result still holds. In fact, we only need to prove that for any $K \leq G$ such that $K \cap \langle x \rangle =1$, $N_G (K \langle x \rangle)=  N_G (K)$. In other words, $K^g \langle x \rangle =K \langle x \rangle \Leftrightarrow g \in N_G (K)$. If $g \in N_G (K)$, it is clear that  $K^g \langle x \rangle =K \langle x \rangle$. On the other hand, if $K^g \langle x \rangle =K \langle x \rangle$, we predicate that $g^{-1} k g \in K$ holds for any $k \in K$. In fact, It follows from $K^g \langle x \rangle =K \langle x \rangle$ that $g^{-1} k g = k_2 x^i$, where $i$ is an integer and $k_2 \in K$. If $x^{k_2} = x$, using the same method in proposition \ref{equivalent 2}, we indicate that $g^{-1} k g \in K$. If $x^{k_2} = x^t$, where $t \not\equiv 1 \,(\!\!\!\mod p)$, and $(t,p)=1$, it follows that 
$$g^{-1} k^{p-1} g = (k_2 x^i)^{p-1} = k_2 ^{p-1} x^{\frac{t^{p-1}-1}{t-1} i} .$$
Since $p \nmid t-1$, it follows from Fermat's Little Theorem that $p | \frac{t^{p-1}-1}{t-1} i$. Hence we conclude that $g^{-1} k^{p-1} g \in K$. Thus by $(p-1,o(k))=1$, we have $g^{-1} k g \in K$. By the choice of $k$, we indicate that $g \in N_G (K)$, and the result follows.

(2) Proposition \ref{equivalent 2} is not necessarily true if $|G|_p > p$. For example, take $G=C_5\times C_3\rtimes D_4$ with generators $G = \langle a,b,c,d ~|~ a^5=b^3=c^4=d^2=1, ab=ba, ac=ca, ad=da, cbc^{-1}=dbd=b^{-1}, dcd=c^{-1} \rangle$. Then there exists a subgroup $C_2 \leq Z(G) = C_{10}$, such that $G/C_2 = S_3 \times C_{10}$, which is a PNC-group. However, one can find that $G$ is not a PNC-group. Thus proposition \ref{equivalent 2} fails in this condition.
\end{remark}
\begin{proposition}\label{Inheritance of NC-subgroup}
Let $G$ be a finite group and $N$ be a normal subgroup of $G$. If $N \leq K$, then $K$ is an NC-subgroup of $G$ if and only if $K/N$ is an NC-group of $G/N$.
\end{proposition}
\begin{proof}
Since $N \leq K$, It follows from $N_G (K) /N = N_{G/N} (K/N)$ and $(K/N)^{G/N} =K^G /N$ that 
$$G/N =  N_{G/N} (K/N) (K/N)^{G/N} \!\!\!\Leftrightarrow\! G/N = \left( N_G (K) /N \right)\left( K^G /N \right)\!\Leftrightarrow\! G/N = N_G (K) K^G /N\!\! \Leftrightarrow \!G = N_G (K) K^G.$$
Thus $K$ is an NC-subgroup of $G$ if and only if $K/N$ is an NC-group of $G/N$, and the proof is completed.
\end{proof}
\begin{remark}
Proposition \ref{Inheritance of NC-subgroup} is not necessarily true if $K \not\leq N$. Take $G = D_4 $ for instance. It is obvious that $C_2$ is not an NC-subgroup of $G$ since $N_G (C_2) = C_2 ^G = C_2 ^2$. Now let $N = \left( C_2 '\right)^2  \unlhd G$, where $\left( C_2 ' \right)^2$ and $C_2 ^2$ are two distinct normal subgroups of $G$. It follows directly that $\left( C_2 ' \right)^2 \cap C_2 =1$. Hence we have $\overline{C_2} = C_2 \cdot N /N = G/N$, which implies that $\overline{C_2}$ is an NC-subgroup of $G/N$. Thus proposition \ref{Inheritance of NC-subgroup} fails in this condition.
\end{remark}
\begin{proposition}\label{Direct product}
Let $G = K \times T$, where $H \leq K $. Then $H$ is an NC-subgroup of $G$ if and only if $H$ is an NC-subgroup of $K$.
\end{proposition}
\begin{proof}
It follows directly that $H^G =( H^K)^T = H^K$. Let $k \in K$ and $t \in T$. A trivial argument shows that 
$$kt \in N_G (H) \Leftrightarrow H^{kt}= H \Leftrightarrow H^k =H \Leftrightarrow  k \in N_K (H). $$
Hence we have $ N_G (H) = T N_K (H)$. It follows directly that 
$$H^G N_G (H) =G \Leftrightarrow  H^K T N_K (H) =G \Leftrightarrow H^K N_K (H) =K .$$
Thus the proof is completed.
\end{proof}
\begin{remark}
Proposition \ref{Direct product} is not necessarily true if we change the condition of direct product into semi-direct product. Let $G = D_4\rtimes S_3$ with generators $G = \langle a,b,c,d ~|~ a^4=b^2=c^3=d^2=1, bab=dad=a^{-1}, ac=ca, bc=cb, dbd=ab, dcd=c^{-1} \rangle$ for example. Then there exists a subgroup $C_2 ^2 \unlhd D_4 \unlhd G$. Hence $C_2 ^2$ is an NC-subgroup of $D_4$. However, it follows from $N_G (C_2 ^2) = C_3 \times D_4$ and $(C_2 ^2)^G =D_4$ that $C_2 ^2$ is not an NC-subgroup of $G$. Therefore proposition \ref{Direct product} fails in this condition.
\end{remark}

\section{Sufficient conditions for solvable PNC-groups}\label{10004}
In this section, we do some computations on certain groups, and obtain some sufficient conditions for a group $G$ to be a solvable PNC-group. In particular, some conditions are even equivalent, and so give us an insight into the structure of solvable PNC-groups.
\begin{proposition}\label{dihedral group}
Let $G$ be a dihedral group, then the following hold:
\begin{itemize}
\item[(1)] All maximal subgroups of $D_n$ are exactly $D_{\frac{n}{p}}$, $p \in \pi(n)$ and $C_n$.
\item[(2)] $D_n $ is a PNC-group if and only if $4 \nmid n$.
\end{itemize}
\end{proposition}
\begin{proof}
(1) Let $D_n = \langle a,b \,|\, a^n = b^2=1,\,bab = a^{-1} \rangle$, and $U$ be a maximal subgroup of $G$. If $U \leq \langle a \rangle \unlhd D_n$, it follows that $U = \langle a \rangle$. If $U$ is not a subgroup of $\langle a \rangle$, it follows from $\langle a \rangle \unlhd D_n$ that $U \langle a \rangle = D_n$. By Lagrange Theorem we have 
$$\frac{|\langle a \rangle||U|}{|\langle a \rangle \cap U|} = \frac{|D_n||U|}{2|\langle a \rangle \cap U|}= |D_n|.$$
Hence $2|\langle a \rangle \cap U| = |U|$. Let $\langle a^r \rangle =\langle a \rangle \cap U$, where $(r,n) \neq  1$. Without the loss of generality, we may assume that $p | (r,n)$, where $p$ is a prime. Now let $g \in U \setminus \langle a \rangle$. Since $D_n = \langle a \rangle \rtimes \langle b \rangle$, there exists $i \in \mathbb{N}$ such that $g = a^i b$. It follows from $g^2 = a^i b a^i b = a^i a^{-i}=1$ that $U = \langle a^r \rangle \rtimes \langle a^i b \rangle$. Since $U \leq \langle a^p \rangle \cdot \langle a^i b \rangle = H$, and $  \langle a^p \rangle \cap \langle a^i b \rangle =1$, again by Lagrange Theorem we have 
$$|H| = \frac{|\langle a^p \rangle||\langle a^i b \rangle|}{|\langle a^p \rangle \cap \langle a^i b \rangle |} =  \frac{2|\langle a \rangle|}{p} = \frac{|D_n|}{p}.$$
Therefore dihedral group $H \cong D_{\frac{n}{p}}$ is a maximal subgroup of $D_n$, which forces $H = U \cong D_{\frac{n}{p}}$.

(2)Let $D_n = \langle a,h \,|\, h^n = a^2=1,\,aha = h^{-1} \rangle$.  Suppose that $D_n$ is a PNC-group with $4|n$, it follows that 
$$\langle a \rangle \ltimes \langle h^4 \rangle = D_{\frac{n}{4}}  \unlhd \langle a \rangle \ltimes \langle h^2 \rangle =D_{\frac{n}{2}} \unlhd \langle a \rangle \ltimes \langle h \rangle = D_n.$$
Hence $D_{\frac{n}{4}} ^{D_n} \leq D_{\frac{n}{2}}$. For $ a h^t \in D_n,\,t,k \in \mathbb{N}$, we have 
$$h^{-t} a \cdot a h^{4k} \cdot a h^{t} = h^{4k-t} a h^{t} =h^{4k-2t} a =a h^{2t-4k}.$$ 
Let $t$ be an odd number, then we conclude that $(a h^{4k} )^{a h^t} \notin D_{\frac{n}{4}}$. Thus $D_{\frac{n}{4}} \ntrianglelefteq D_n$, which implies that $N_{D_n} (D_{\frac{n}{4}}) \leq  D_{\frac{n}{2}}$. Therefore we indicate that $N_{D_n} (D_{\frac{n}{4}}) D_{\frac{n}{4}} ^{D_n} \leq D_{\frac{n}{2}}$,
A contradiction to the fact that $D_{\frac{n}{4}}$ is an NC-subgroup of $D_n$. Thus we conclude that $4 \nmid n$.

Assume that $4 \nmid n$. For any subgroup  $H$ of $D_n$, either there exists an integer $t$ such that $H = \langle h^t \rangle$, or $H = \langle ah^j \rangle \ltimes \langle h^t \rangle,\,j \in \mathbb{N}$. For the first case, one can easily find that $\langle h^t \rangle \unlhd D_n$. For the second case, without the loss of generality, we may assume that $j=0$, i.e. $H = \langle a \rangle \ltimes \langle h^t \rangle$. Since for $k,s \in \mathbb{N}$, $h^{-k} a \cdot a h^{ts} \cdot a h^k = h^{ts -2k} a =a h^{2k-ts} ,\,h^{-k} \cdot a h^{ts} \cdot h^k =a h^{2k+ts}$, we indicate that
$$\left( \langle a \rangle \ltimes \langle h^t \rangle \right)^{D_n} =\langle a \rangle \ltimes \left( \langle h^t \rangle \times \langle h^2 \rangle \right).$$
If $(2,t)=1$ or $(2,n)=1$, then we conclude that 
$$\left( \langle a \rangle \ltimes \langle h^t \rangle \right)^{D_n} =\langle a \rangle \ltimes \left( \langle h^t \rangle \times \langle h^2 \rangle \right)=D_n.$$
Hence $H$ is a NC-subgroup of $D_n$. If $(2,t) \neq 1$ and $2 \parallel n$, without the loss of generality we may assume that $(2,t)=2$. Then we conclude that $\left( \langle a \rangle \ltimes \langle h^t \rangle \right)^{D_n} =D_{\frac{n}{2}}$. Let $k = \frac{t}{2}$, it follows that $h^{-k} a \cdot a h^{ts} a h^k =a h^{2k-ts} =a h^{t(1-s)},\,s \in \mathbb{N}$. Hence we have $a h^{\frac{t}{2}} \in N_{D_n} \left( \langle a \rangle \ltimes \langle h^t \rangle \right)$, which implies that 
$$N_G (H) H^G = N_{D_n} \left( \langle a \rangle \ltimes \langle h^t \rangle \right) \left( \langle a \rangle \ltimes \langle h^t \rangle \right)^{D_n} =G$$ 
Thus $H$ is an NC-subgroup of $D_n$. By the choice of $H$, we indicate that $D_n$ is a PNC-group. 
\end{proof}

\begin{proposition}
Let $G = \langle a,b: a^{2} = b^{n},\,o(a)= 4,\,o(b) =2n,\,a^{-1} b a = b^{-1} \rangle$, where $n \in \mathbb{N^{*}}$. Then $G$ is a PNC-group if and only if $4 \nmid n$.
\end{proposition}
\begin{proof}
For $i,j \in \mathbb{N}$, it is easy to find that:
$$\left( a^{i} \right) ^{b^{j}}  = a^i b^{j((-1)^{i+1} +1)}.$$ 
If $4 | n$, let $U= \langle ab^2 \rangle$. For $u,v \in \mathbb{N}$, one can easily find that 
$$(a b^2)^{a^u b^v} = (a b^{2 \cdot (-1)^{u}})^{b^v} = a b^{2 \cdot (-1)^{u} + 2v},\,(a^{-1} b^2)^{a^u b^v} = a^{-1} b^{2 \cdot (-1)^{u} + 2v}. $$
For $a^{i} b ^{2x},a^j b ^{2y}$, where $i,j,x,y$ are integers, it follows that 
$$a^{i} b ^{2x} a^j b ^{2y}  = a^{i+j} b ^{2x \cdot (-1)^{j} +2y}.$$
Thus we indicate that $U^G = \langle a \rangle \cdot \langle b^2 \rangle$. Now let $a^i b^j \in N_G (U)$, then we have
\begin{align*}
(a b^2)^{a^i b^j} = a b^{2 \cdot (-1)^{i} + 2j} = a b^2 ~\text{or}~  (a b^2)^{a^i b^j} = a b^{2 \cdot (-1)^{i} + 2j} = a^{-1} b^2 .
\end{align*}
The first case implies that $2|j$, the second case indicates that $2 |j$ as well. Hence we conclude that $N_G (U) \leq \langle a \rangle \cdot \langle b^2 \rangle$ Thus $N_G (U) U^G \leq \langle a \rangle \cdot \langle b^2 \rangle$, i.e. $G$ is not a PNC-group. 

\!\!\!\!\!\!\!\!\!\!\!\!\!\! If $n$ is odd, it follows that $G = \langle b^2 \rangle \rtimes \langle a \rangle$. Clearly $a$ induces a power automorphism on $\langle b^2 \rangle$. For $i \in \mathbb{N}$ where $i|n$, we have 
$$(b^{2i})^{a^j} = b^{2i \cdot (-1)^{j}} = b^{2i} ~\text{or}~ (b^{2i})^{\frac{n-i}{i}} .$$
The second case suggests that $(\frac{n-i}{i}-1, \frac{n}{i})=1$, hence we conclude from theorem \ref{sufficiency} that $G$ is a solvable PNC-group.

\!\!\!\!\!\!\!\!\!\!\!\!\!\! For $2 || n$, the case of $U \leq \langle b \rangle$ is clear. Suppose that $U \not\leq \langle b \rangle$, then Lagrange Theorem shows that $U = \langle b^k \rangle \cdot \langle a b^j \rangle$ for some $j \in \mathbb{N}$, where $v_2 (j) \leq 1$. Since $N_G (\langle a b^j \rangle) \leq N_G (\langle U \rangle)$ and $(\langle a b^j \rangle)^G \leq U^G$, Therefore if suffices to prove that $\langle a b^j \rangle$ is an NC-subgroup of $G$. If $2|j$, since $(a b^{j})^b = a b^{j+2}$, one can easily get that $\langle b^2 \rangle \leq \langle a b^j \rangle ^G$. Hence $\langle a \rangle \leq \langle a b^j \rangle ^G$. Now we conclude that 
$$(a b^j)^{ b^{1.5n}} = a b^{j+ 3n} = a^{-1} b^j. $$
Hence $b^{1.5n} \in N_G (\langle a b^j \rangle)$. Since $1.5n$ is odd, we indicate that $\langle b \rangle \leq N_G (\langle a b^j \rangle) \langle a b^j \rangle ^G$.  It follows that $\langle a b^j \rangle ^G \geq \langle b \rangle \cdot \langle a \rangle = G$. Thus $\langle a b^j \rangle$ is an NC-subgroup of $G$, and the proof is completed. 
\end{proof}

\begin{question}
Is $S_n$ a PNC-group for any $n \in \mathbb{N}$ except for $n=4$?
\end{question}

Now we will study a few kinds of solvable groups which are similar to those in the third section. Actually we predicate some necessary and sufficient conditions for several kinds of solvable groups. 

In this section, we may assume that $G = A \rtimes \langle a \rangle$ where $o(a) =p^{\alpha}$, $A = C_{p_1 ^{\alpha _1}} \times C_{p_2 ^{\alpha _2}} \times \cdots \times C_{p_n ^{\alpha _n}}$ with $p,p_1,\cdots,p_n$ being distinct prime numbers. One can easily find that every subgroup of $A$ is normal in $G$. Now, we consider the action induced by $a$ on $C_{p_i ^{\alpha _i}},\,i=1,2,\cdots,n$. Let $C_{p_i ^{\alpha _i}} = \langle a_i \rangle$, where $o(a_i) = p_i ^{\alpha _i},\,i=1,2,\cdots,n$. It follows that $a$ induces a power automorphism on $A$, and we have
$$({a_i ^{-1}} )^{a} = a_i ^{t_i} = a^{-1} a_i ^{-1} a,\,t_i \in \mathbb{N}, t_i \leq p_i ^{\alpha_i }-1,\,(t_i,p_i)=1.$$

\begin{lemma}
Let $s,k_1,k_2,\cdots,k_n \in \mathbb{N}$, then we have
$$(a^s)^{a_1 ^{k_1} a_2 ^{k_2} \cdots a_n ^{k_n}} = a^s (a_n ^{k_n})^{1-(-t_n)^s} \cdots (a_2 ^{k_2})^{1-(-t_2)^s} (a_1 ^{k_1})^{1-(-t_1)^s}.$$
\end{lemma}
\begin{proof}
Since $(a_i ^{-1})^a = a_i ^{t_i} = a^{-1} a_i ^{-1} a$, it follows that $a^{a_i} = a_i ^{-1} a a_i = a a_i ^{t_i} a_i = a a_i ^{t_i +1}$, and $a_i a = a a_i ^{-t_i}$. Therefore we conclude that 
$$a^{a_i ^{k_i}} = (a a_i ^{t_i +1})^{a_i ^{k_i -1}} = (a a_i ^{t_i +1} a_i ^{t_i +1})^{a_i ^{k_i -2}} = \cdots = a a_i ^{k_i (t_i +1)}.$$
Induction on $n$ gives that $ (a^n)^{a_i} = a^n a_i ^{1 - (-t_i)^n},\,n \in \mathbb{N}$. Hence for any $s \in \mathbb{N}$, we indicate that 
$$(a^s)^{a_i ^{k_i}} = (a^s a_i ^{1 - (-t_i)^s})^{a_i ^{k_i -1}} = (a^s a_i ^{1 - (-t_i)^s} a_i ^{1 - (-t_i)^s} )^{a_i ^{k_i -2}}=  \cdots = a^s  
(a_i ^{1 - (-t_i)^s})^{k_i} = a^s (a_i ^{k_i})^{1 - (-t_i)^s}.$$
Apply the observation above, it follows that 
\begin{align*}
&(a^s)^{a_1 ^{k_1} a_2 ^{k_2} \cdots a_n ^{k_n}} =( a^s (a_1 ^{k_1})^{1 - (-t_1)^s} )^{a_2 ^{k_2} \cdots a_n ^{k_n}} = (a^s (a_2 ^{k_2})^{1 - (-t_2)^s}  (a_1 ^{k_1})^{1 - (-t_1)^s})^{a_3 ^{k_3} \cdots a_n ^{k_n}} \\
 =&\cdots = a^s (a_n ^{k_n})^{1-(-t_n)^s} \cdots (a_2 ^{k_2})^{1-(-t_2)^s} (a_1 ^{k_1})^{1-(-t_1)^s}. \qedhere
\end{align*}
\end{proof}
We shall freely use this lemma in this section.
\begin{lemma}\label{2}
Let $s_1,s_2,f_1,\cdots,f_n,u_1,\cdots,u_n$ be non-negative integers. Assume that 
$$g_1 = a^{s_1} a_n ^{f_n} \cdots a_1 ^{f_1} ,\,g_2 =  a^{s_2} a_n ^{u_n} \cdots a_1 ^{u_1}.$$
Then we have  
$$g_1 g_2 = a^{s_1 +s_2} a_n ^{u_n +(-t_n)^{s_2} f_n} \cdots a_2 ^{u_2 +(-t_2)^{s_2} f_2} a_1 ^{u_1 +(-t_1)^{s_2} f_1}.$$
\end{lemma}
\begin{proof}
Since $a_i a = a a_i ^{-t_i}$, it follows from $a_i ^{f_i} a^{s_2} = a a_i ^{-t_i f_i} a^{s_2 -1} = \cdots = a^{s_2} a_i ^{(-t_i)^{s_2} f_i}$ that
\begin{align*}
g_1 g_2 &= a^{s_1} a_n ^{f_n} \cdots a_1 ^{f_1} a^{s_2} a_n ^{u_n} \cdots a_1 ^{u_1} = a^{s_1} a_n ^{f_n} \cdots a_2 ^{f_2} a^{s_2} a_1 ^{(-t_1)^{s_2} f_1} a_n ^{u_n} \cdots a_1 ^{u_1} \\
&=  a^{s_1} a_n ^{f_n} \cdots a_3 ^{f_3} a^{s_2} a_1 ^{(-t_1)^{s_2} f_1} a_2 ^{(-t_2)^{s_2} f_2} a_n ^{u_n} \cdots a_1 ^{u_1}  =\cdots \\
&= a^{s_1 +s_2}  a_1 ^{(-t_1)^{s_2} f_1} a_2 ^{(-t_2)^{s_2} f_2} \cdots  a_n ^{(-t_n)^{s_2} f_n} a_n ^{u_n} \cdots a_1 ^{u_1} \\
&= a^{s_1 +s_2} a_n ^{u_n +(-t_n)^{s_2} f_n} \cdots a_2 ^{u_2 +(-t_2)^{s_2} f_2} a_1 ^{u_1 +(-t_1)^{s_2} f_1}. \qedhere
\end{align*}
\end{proof}
Now we introduce the following notations to simplify the narration .
\begin{definition}
Let $n$ be a positive integer, and $n = p_1 ^{\alpha_1} p_2 ^{\alpha_2} \cdots  p_m ^{\alpha_m}$, where $p_1,p_2,\cdots,p_m$ are distinct primes. Set $\pi (n) = \lbrace p_1,p_2,\cdots,p_n \rbrace$. For a prime $p$, $v_p (n) := \alpha_i$, if $p =p_i,\,i=1,2,\cdots,m$; $v_p (n) := 0$, if $p \notin \pi (n)$.
\end{definition}

\begin{theorem}\label{3}
Suppose that $p > p_1,p_2,\cdots,p_n$. Then $G$ is a PNC-group if and only if 
$$v_{p_i} (t_i +1) = 0 ~\text{or}~ \alpha_i,\,i=1,2,\cdots,n.$$
\end{theorem}
\begin{proof}
The proof is proceeded via the two cases.
\begin{itemize}
\item[(1)] $\alpha_1 = \alpha_2 = \cdots = \alpha_n =1$. 
\end{itemize}
It suffices to prove that the conclusion holds whatever $t_i,i=1,2,\cdots,n$ are. Consider the subgroup $\langle a \rangle$ firstly. For $s,k_1,k_2,\cdots,k_n \in \mathbb{N}$, it follows that 
$$(a^s)^{a_1 ^{k_1} a_2 ^{k_2} \cdots a_n ^{k_n}} = a^s (a_n ^{k_n})^{1-(-t_n)^s} \cdots (a_2 ^{k_2})^{1-(-t_2)^s} (a_1 ^{k_1})^{1-(-t_1)^s}.$$
Without loss of generality, $1+ t_i \equiv 0 \,(\!\!\!\mod p_i),\,i=1,2,\cdots,j$, $1+t_i \not\equiv 0\,(\!\!\!\mod p_i),\,i=j+1,\cdots,n$. For any $w \in \lbrace  j+1,\cdots, n \rbrace$, let $s =1, k_i =0 , i \neq  w$, we conclude that 
$$(a^s)^{a_1 ^{k_1} a_2 ^{k_2} \cdots a_n ^{k_n}} = a(a_w ^{k_w})^{1+ t_w} \in \langle a \rangle ^G.$$
Since $1 + t_w \not\equiv 0 \,(\!\!\! \mod p_w )$, by the choice of $k_w$, $C_{p_w ^{\alpha_w}} = \langle  a_w \rangle \leq \langle a \rangle ^G$, $w = j+1,\cdots,n$. For any $h \in \lbrace 1,2,\cdots, j \rbrace$, since $1+ t_h \equiv 0 \,\,(\!\!\!\mod p_h)$, it follows that
$$ (a^s)^{a_h ^{k_h}} = a^s (a_h ^{k_h})^{1+ t_h} = a^s. $$
By the choice of $s$, we indicate that $a_h ^{k_h} \in N_G (\langle a \rangle)$. It follows from the randomness of $k_h$ that $C_{p_h ^{\alpha_h}} \leq N_G (\langle a \rangle)$, $ h \in \lbrace 1,2,\cdots, j \rbrace$. Hence $C_{p_i ^{\alpha_i}} \leq N_G (\langle a \rangle) \langle a \rangle ^G$, $i =1,2,\cdots,n$, which implies that $N_G (\langle a \rangle) \langle a \rangle ^G = G$. Let $U$ be a subgroup of $G$. If $U$ has trivial $p$-part, it follows directly that $U \unlhd G$. If $U$ has non-trivial $p$-part, let $P \in {\rm{Syl}}_p (U)$. Without loss of generality, we may assume that $P  = \langle a^q \rangle \leq  \langle a \rangle$, where $q = p^{\alpha_0}$. Now we predicate that $\langle a \rangle ^G =\langle a^{q} \rangle  ^G \langle a \rangle$. As a matter of fact, we have
\begin{align*}
\langle a^{q} \rangle ^G &= \langle  (a^{sq})^{a_1 ^{k_1} a_2 ^{k_2} \cdots a_n ^{k_n}}| s,k_1,k_2,\cdots,k_n \in \mathbb{N} \rangle \\
&=\langle a^{sq} (a_n ^{k_n})^{1-(-t_n)^{sq}} \cdots (a_2 ^{k_2})^{1-(-t_2)^{sq}} (a_1 ^{k_1})^{1-(-t_1)^{sq}} | s,k_1,k_2,\cdots,k_n \in \mathbb{N} \rangle.
\end{align*}
Since $p_i \nmid 1,\,p_i \nmid t_i$, it follows from lifting-the-exponent lemma that $v_{p_i} (1-(-t_i)^{q}) = v_{p_i} (1+t_i) +v_{p_i} (q) = v_{p_i} (1+t_i)$, if $p_i | 1+t_i$. If $p_i \nmid 1+t_i$, but $p_i | 1-(-t_i)^{q}$, then we conclude that $p_i | 1-(-t_i)^{(q,p_i -1)}$. Since $p > p_1,p_2,\cdots,p_n$, we have $p_i | 1+t_i$, a contradiction. Thus  $v_{p_i} (1-(-t_i)^{q}) = v_{p_i} (1+t_i) +v_{p_i} (q) = v_{p_i} (1+t_i)$ holds when $p_i \nmid 1+t_i$. If $v_{p_i} (1+t_i) = 0$, then we indicate that $ C_{p_i ^{\alpha_i}}  \leq \langle a \rangle ^G $, and $v_{p_i} (1-(-t_i)^{q}) = 0 $. It follows that $C_{p_i ^{\alpha_i}}  \leq \langle a \rangle ^G$ and $C_{p_i ^{\alpha_i}}  \leq \langle a^{q} \rangle ^G$. If $v_{p_i} (1+t_i) \neq 0$, then we indicate that $v_{p_i} (1+t_i)  =v_{p_i} (1-(-t_i)^{q})=1$. Therefore we have $(a_i ^{k_i})^{1-(-t_i)^{s}}=(a_i ^{k_i})^{1-(-t_i)^{qs}} =1 ,\,s,k_i \in \mathbb{N}$. It follows from lemma \ref{2} that $1 =  C_{p_i ^{\alpha_i}} \cap \langle a \rangle ^G =  C_{p_i ^{\alpha_i}} \cap \langle a^{q} \rangle ^G $. Hence we indicate that $\langle a \rangle ^G =\langle a^{q} \rangle  ^G \langle a \rangle \leq U^G \langle a \rangle$. Now we may assume that $U = \langle a_{d_1} \rangle \times \langle a_{d_2} \rangle \times \cdots \times \langle a_{d_r} \rangle \rtimes  \langle a^{q} \rangle$, where $1 \leq d_1,\cdots,d_r \leq n$ are distinct integers. Since $\langle a_{d_i} \rangle $ are normal subgroups of $G$, it follows that $\langle a \rangle \leq N_G (\langle a \rangle) \leq N_G (U)$. Thus $ N_G (U)(\langle a \rangle U^G) = (N_G (U) \langle a \rangle) U^G =N_G (U)  U^G \geq N_G (\langle a \rangle) \langle a \rangle ^G = G$. By the choice of $U$, we conclude that $G$ is a PNC-group.
\begin{itemize}
\item[(2)] There exists $i \in \mathbb{N}$ such that $\alpha_i \neq 1$. Without loss of generality we may assume that $\alpha_1 = \cdots \alpha_j =1$, $\alpha_{j+1}, \cdots, \alpha_n \neq 1$. 
\end{itemize}
\begin{itemize}
\item[(i)] Necessity of the proof.
\end{itemize}
Suppose that $G$ is a PNC-group, but there exists $i_0 \in \mathbb{N^{*}}$ such that $1 \leq v_{p_{i_0}} (1+ t_{i_0}) < \alpha_{i_0}$. It is clear that $j+1 \leq i_0 \leq n$. Now consider the group $\langle a \rangle ^G$. For $s,k_1,k_2,\cdots,k_n \in \mathbb{N}$, we have
$$(a^s)^{a_1 ^{k_1} a_2 ^{k_2} \cdots a_n ^{k_n}} = a^s (a_n ^{k_n})^{1-(-t_n)^s} \cdots (a_2 ^{k_2})^{1-(-t_2)^s} (a_1 ^{k_1})^{1-(-t_1)^s}.$$
It follows that 
$$\langle a \rangle ^G = \langle (a^s)^{a_1 ^{k_1} a_2 ^{k_2} \cdots a_n ^{k_n}} | s,k_1,k_2,\cdots,k_n \in \mathbb{N} \rangle.$$
On the one hand, we have $p_{i_0} | 1+ t_{i_0} | 1-(-t_{i_0})^s | k_i (1-(-t_{i_0})^s)$. It follows from lemma \ref{2} that for any $g = a^{l} a_n ^{l_n} \cdots a_2 ^{l_2} a_1 ^{l_1} \in \langle a \rangle ^G$, $p_{i_0} | l_{i_0}$.
On the other hand, for any $g = a^{k} a_n ^{k_n} \cdots a_2 ^{k_2} a_1 ^{k_1} \in N_G (\langle a \rangle)$, we indicate that
$$a^g = a (a_n ^{k_n})^{1-(-t_n)} \cdots (a_2 ^{k_2})^{1-(-t_2)} (a_1 ^{k_1})^{1-(-t_1)} \in \langle a \rangle.$$
Since $G = A \rtimes \langle a \rangle$, we have $(a_n ^{k_n})^{1-(-t_n)} \cdots (a_2 ^{k_2})^{1-(-t_2)} (a_1 ^{k_1})^{1-(-t_1)} =1$. It follows from $A = C_{p_1 ^{\alpha _1}} \times C_{p_2 ^{\alpha _2}} \times \cdots \times C_{p_n ^{\alpha _n}}$ that $  (a_1 ^{k_1})^{1-(-t_1)} = (a_2 ^{k_2})^{1-(-t_2)} = \cdots (a_n ^{k_n})^{1-(-t_n)} =1$. Therefore we have $a_{i_0} ^{k_{i_0}(1+t_{i_0})} =1$. Hence $p_{i_0} ^{\alpha_{i_0}} | k_{i_0} (1+t_{i_0})$, which implies that $1 \leq v_{p_{i_0}} (k_{i_0})$, i.e. $p_{i_0} | k_{i_0}$. By lemma \ref{2}, we conclude that the exponent of $a_i$ in the factorization of any element in $N_G (\langle a \rangle) \langle a \rangle ^G$ is a multiple of $p_{i_0}$. Thus we indicate that $C_{{p_{i_0}}^{\alpha _{i_0}}} \nleqslant N_G (\langle a \rangle) \langle a \rangle ^G$, a contradiction to the fact that $\langle a \rangle$ is an NC-group of $G$.
\begin{itemize}
\item[(ii)] Sufficiency of the proof.
\end{itemize}
Without loss of generality, we may assume that $v_{p_i} (t_i +1) = 0,\,i=1,\cdots,l$, $v_{p_i} (t_i +1)  = \alpha_i,\,i=l+1,\cdots,n$. On the one hand, let $a_i ^{k_i} \in C_{{p_i}^{\alpha_i}},\,i=1,\cdots,l$, then we indicate that 
$$a^{a_i ^{k_i}} = a (a_i ^{k_i})^{1+t_i}.$$
By the choice of $k_i$, it follows from $(1+t_i, p_i) =1$ that $C_{{p_i}^{\alpha_i}} =  \langle a_i ^{1+t_i} \rangle \leq \langle a \rangle ^G$. On the other hand, let $a_i ^{k_i} \in C_{{p_i}^{\alpha_i}},\,i=l+1,\cdots,n$. We conclude from $p_i ^{\alpha_i} | 1+t_i$ that 
$$a^{a_i ^{k_i}} = a (a_i ^{k_i})^{1+t_i} = a.$$
By the choice of $k_i$ we have $C_{{p_i}^{\alpha_i}} \leq N_G (\langle a \rangle)$, $i=l+1,\cdots,n$. Thus we indicate that $G = N_G (\langle a \rangle) \langle a \rangle ^G$. Now let $U$ be a subgroup of $G$ and $P$ be a Sylow $p$-subgroup of $U$. If $U$ has trivial $p$-part, then $U \unlhd G$. If $U$ has non-trivial $p$-part, without loss of generality we may assume that $P = \langle a^{q} \rangle \leq  \langle a \rangle$ with $q = p^{\alpha_0}$. It follows from the same method used in (1) that $U^G  \langle a \rangle \geq \langle a \rangle ^G$. Since all subgroups of $A$ are normal in $G$, we indicate that $\langle a \rangle \leq N_G (\langle a \rangle) \leq N_G (U)$. Thus $ N_G (U)(\langle a \rangle U^G) = (N_G (U) \langle a \rangle) U^G =N_G (U)  U^G \geq N_G (\langle a \rangle) \langle a \rangle ^G = G$. By the choice of $U$, we conclude that $G$ is a PNC-group.
\end{proof}
\begin{theorem}\label{sufficiency}
Assume that $G$ is a finite group such that $G = A \rtimes D$, where $A$ is an abelian Hall subgroup and $D$ is a Dedkind group. If every subgroup of $A$ is normal in $G$, and for any $a \in A$, $d \in D$, $a^d = a^n$ with $(n,o(a)) =1$, either $n \equiv 1 \,(\!\!\!\mod o(a))$ or $(n-1,o(a))=1$, then $G$ is a PNC-group. In particular, for any $a \in A$ and any subgroup $H$, either $a \in N_G (H)$ or $a \in H^G$.
\end{theorem}
\begin{proof}
Let $H$ be a subgroup of $G$. For any $p_i \in \pi (A) = \lbrace p_1,p_2,\cdots,p_n \rbrace,\,i=1,2,\cdots,n$, one can easily indicates that $G$ is $p_i$-closed. Let $P_i \leq H$ be a Sylow $p_i$-subgroup of $H$, then $P_i \leq S$ where $S$ denotes the  Sylow $p_i$-subgroup of $G$. Therefore $P_i$ is normal in $G$. Now let $P:=P_1 \times P_2 \times \cdots \times P_n \unlhd G$, then $P$ is a Hall subgroup of $A$. Hence by Schur-Zassenhaus theorem we conclude that $P$ has a complement $K$ in $A$. Since $G/A \cong D$ is solvable, $A$ is solvable, we conclude that $G$ is solvable. Hence $\pi$-Sylow theorem holds for $G$. Thus without the loss of generality, we may assume that $K \leq D$. Now let 
$$P = \langle p_{1,1} \rangle \times \langle p_{1,2} \rangle \times \cdots \times \langle p_{1,t_1} \rangle \times \langle p_{2,1} \rangle \times \cdots \times \langle p_{n,t_n} \rangle,\,o(p_{i,j}) = p_i ^{\alpha_{i,j}},\,1 \leq j \leq t_i,\,i=1,2,\cdots,n.$$ 
Therefore we have 
$$H= P \rtimes K = \langle p_{1,1} \rangle \times \langle p_{1,2} \rangle \times \cdots \times \langle p_{1,t_1} \rangle \times \langle p_{2,1} \rangle \times \cdots \times \langle p_{n,t_n} \rangle \rtimes K.$$
For any $d \in D$, it follows from $\langle p_{i,j} \rangle \unlhd G,\,1 \leq j \leq t_i,\,i=1,2,\cdots,n$ that
\begin{align*}
H^d &= \langle p_{1,1} \rangle^d \times \langle p_{1,2} \rangle^d \times \cdots \times \langle p_{1,t_1} \rangle^d \times \langle p_{2,1} \rangle^d \times \cdots \times \langle p_{n,t_n} \rangle^d \rtimes K^d \\
&= \langle p_{1,1} \rangle \times \langle p_{1,2} \rangle \times \cdots \times \langle p_{1,t_1} \rangle \times \langle p_{2,1} \rangle \times \cdots \times \langle p_{n,t_n} \rangle \rtimes K =H.
\end{align*}
Thus $d \in N_G (H)$. Hence we conclude that $D \leq N_G (H)$.
Now for any $a \in A$ and $k \in K \leq D$, it follows that 
\begin{align*}\label{shit}
k = a^{n_{a,k}},\,k^a = k a^{1-n_{a,k}},\,n_{a,k} \in \mathbb{N}.
\tag{1}
\end{align*}
If $n_{a,k} \equiv 1 \,(\!\!\!\mod o(a))$ holds for any $k \in K$, then $k^a = k$ holds for any $k \in K$. Thus $a \in N_G (H)$. If there exists $k \in K$ such that $n_{a,k} \not\equiv 1 \,(\!\!\!\mod o(a))$, by the hypothesis we have $(n_{a,k} -1,o(a))=1$. And by (\ref{shit}) we indicate that $\langle a \rangle \leq H^G$. Hence we conclude that $A \leq N_G (H) H^G$, and therefore $A D = G \leq N_G (H) H^G$. Thus $H$ is an NC-subgroup of $G$. By the choice of $H$, $G$ is a PNC-group.
\end{proof}

\section{Characterizations for solvable PNC-groups}\label{10005}
In the last section, we give some characterizations for non-PE-groups whose proper subgroups are solvable PNC-groups, non-PE-groups whose proper subgroups are solvable PNC-groups, non-PE-groups whose proper subgroups are ON-groups, and classify the ON-groups. Characterizations for groups whose maximal subgroups are solvable PNC-groups are also given in the section. In particular, we classify the non-abelian simple groups whose second maximal subgroups are solvable PNC-groups, and give criteria for non-solvable groups whose second maximal subgroups are solvable PNC-groups. 
\begin{theorem}\label{classification}
If G is a finite non-PE-group, each of whose proper subgroup is a 
solvable PNC-group, then $|\pi (G)| \leq 2$ and one of the following statements is true:
\begin{itemize}
\item[(1)] $G \cong D_4$.
\item[(2)] $G = \langle a,x : a^{p^n} = x^p =1,\,x^{-1} a x = a^{1+p^{n-1}} \rangle$, where $p$ is an odd prime and $n \geq 2$.
\item[(3)] $G = \langle a,b,x: a^{p^n} = b^p = x^p =1,\,[x,a]=b,\,[a,b] = [b,x] =1 \rangle$, where $p$ is a prime.
\end{itemize}
In the following (4)-(6), $P \in {\rm{Syl}}_p (G)$, $Q \in {\rm{Syl}}_p (G)$, $p < q$ are primes. One can easily find that (4)-(6) are solvable groups. It follows from proposition \ref{normal complement} that any proper subgroup of $G$ is $p$-nilpotent. 
\begin{itemize}
\item[(4)] $G = PQ$ is supersolvable, where $P$ is cyclic and $Q$ is elementary abelian of order $q^2$ with $O_q (G) \neq 1$.
\item[(5)] $G =PQ$ is a minimal non-nilpotent group with $P \unlhd G$, and one of the following is true:
\begin{itemize}
\item[(i)] $G$ is a minimal non-abelian group of order $q^m p^n$, and we have
\begin{align*}
G=&\langle a,b_1,b_2, \cdots,b_n|1=a^{q^m}=b_1 ^p=b_2 ^p= \cdots =b_n ^p,\,b_i b_j =b_j  b_i,\,i,j =1,2,\cdots,n,\\
& b_i ^a =b_{i+1},\,i,j =1,2,\cdots,n-1,b_n ^a=b_1 ^{d1} b_2 ^{d_2} \cdots b_n ^{d_n}  \rangle,
\end{align*}
where $f(x) =x^n-d_n x^{n-1}- \cdots -d_2 x-d_1$ is irreducible  on $F_p$, which is a factor of $ x^{q^m}-1$.
\item[(ii)] $G \cong {\rm SL}(2,3)$.
\end{itemize}
\item[(6)] $G =PQ$ is a minimal non-supersolvable group with $Q \unlhd G$, where $Q$ is elementary abelian of order $|Q| > q$, and $P$ is cyclic ,acting irreducibly on $Q$. 
\end{itemize}
\end{theorem}
\begin{proof}
It follows from proposition \ref{NE} (10) that any proper subgroup of $G$ is a PE-group. By {{\cite[Theorem 2]{SL}}}, one of the following statements is true:
\begin{itemize}
\item[(1)] $G \cong D_4$.
\item[(2)] $G = \langle a,x : a^{p^n} = x^p =1,\,x^{-1} a x = a^{1+p^{n-1}} \rangle$, $p$ an odd prime, $n \geq 2$.
\item[(3)] $G = \langle a,b,x: a^{p^n} = b^p = x^p =1,\,[x,a]=b,\,[a,b] = [b,x] =1 \rangle$ for some prime $p$.
\end{itemize}
In the following (4)-(6), $P \in {\rm{Syl}}_p (G)$, $Q \in {\rm{Syl}}_p (G)$, $p < q$ are distinct prime numbers. 
\begin{itemize}
\item[(4)] $G = PQ$ is supersolvable, where $P$ is cyclic, $Q$ is elementary abelian of order $q^2$.
\item[(5)] $G =PQ$ is minimal non-nilpotent group with $P \unlhd G$, $P$ is elementary abelian and $Q$ is cyclic or $P$ is an ultraspecial 2-group and $|Q| =q$. 
\item[(6)] $G =PQ$ is a minimal non-supersolvable group with $Q \unlhd G$, where $Q$ is elementary abelian of order $|Q| > q$, and $P$ is cyclic ,acting irreducibly on $Q$. 
\end{itemize}

(1) One can easily find that $D_4$ is a non-PE-group, each of whose proper subgroup is a solvable PNC-group.

(2) We predicate that $\langle x \rangle$ is not an NE-subgroup of $G$. Since $a^{-1} x = x a^{-1-p^{n-1}}$, it follows that 
$$x^{a^i} = a^{-i+1} ( a^{-1} x a ) a^{i-1} = a^{-i+1} x a^{-p^{n-1}} a^{i-1} = \cdots = xa^{-ip^{n-1}}.$$
By the choice of $i$ we indicate that $\langle a^{p^{n-1}} \rangle \leq \langle x \rangle ^G$. As $x^{a^p} = x a^{-p^n} = x$, we conclude that $N_G (\langle x \rangle) \geq \langle a^p \rangle$. Hence we have $N_G (\langle x \rangle)  \cap \langle x \rangle ^G \geq \langle a^p \rangle \langle x \rangle > \langle x \rangle$. Thus the minimal subgroup $\langle x \rangle$ is not a NE-subgroup of $G$. Therefore $G$ is not a PE-subgroup. By {{\cite[Theorem (4) 2)]{SC}}} we indicate that $G$ is a minimal non-abelian group, which implies that any proper subgroup of $G$ is abelian. Hence any proper subgroup of $G$ is a solvable PNC-group.

(3) We predicate that $\langle x \rangle$ is not an NE-subgroup of $G$. Since $x ^a = xb$, we have $\langle b \rangle \leq \langle x \rangle ^G$. As $\langle b \rangle \leq Z(G) $, we conclude that $N_G (\langle x \rangle) \cap \langle x \rangle ^G \geq \langle b \rangle \langle x \rangle > \langle x \rangle$, so the result holds. Again by {{\cite[Theorem (4) 1)]{SC}}}, we conclude that all proper subgroups of $G$ are Dedekind groups. Hence any proper subgroup of $G$ is a solvable PNC-group.

(4) Since $G$ is solvable, there exists a normal subgroup of $U$ such that $|G / U|$ is a prime. It follows from $q \mid |U|$ that $U$ has non-trivial $q$-part. Let $ Q_1$ be a Sylow $q$-subgroup of $U$. As $U$ is $p$-nilpotent, we have $Q_1 ~{\rm{char}}~ U \unlhd G$. Thus we conclude that $O_q (G) \neq 1$.

(5) Since $G$ is minimal non-nilpotent group, it follows from Lemma \ref{nilpotent} that $G$ is a minimal non-Dedekind group. As $\pi (G) =2$, by {{\cite[Theorem (2)(3)]{SC}}}, we indicate that $G$ is isomorphic to one of the following groups:
\begin{itemize}
\item[(i)] $G$ is a minimal non-abelian group of order $q^m p^n$, and we have
\begin{align*}
G=&\langle a,b_1,b_2, \cdots,b_n|1=a^{q^m}=b_1 ^p=b_2 ^p= \cdots =b_n ^p,\,b_i b_j =b_j  b_i,\,i,j =1,2,\cdots,n,\\
& b_i ^a =b_{i+1},\,i,j =1,2,\cdots,n-1,b_n ^a=b_1 ^{d1} b_2 ^{d_2} \cdots b_n ^{d_n}  \rangle,
\end{align*}
where $f(x) =x^n-d_n x^{n-1}- \cdots -d_2 x-d_1$ is irreducible  on $F_p$, which is a factor of $ x^{q^m}-1$.
\item[(ii)] $G$ is an inner 3-closed group, and we have
$$G=\langle a,b,c|a^{3^m}=b^4=c^4=1,b^2=c^2,cb=b^{-1}c,a^{-1}ba=c,a^{-1}ca=cb \rangle.  $$
\end{itemize}
For (ii), since the Sylow 2-subgroup of $G$ is not elementary abelian, we conclude that $m=1$, i.e. $|G|=24$. It follows directly from enumeration that only ${\rm SL}(2,3)$ satisfies the condition.

(6) It follows directly from {{\cite[Theorem 2 (f)]{SL}}}.
\end{proof}
\begin{theorem}
If $G$ is a finite non-PE-group, each of whose proper subgroup is an ON-group, then one of the following statements is true:
\begin{itemize}
\item[(1)] $G$ is a minimal non-abelian group of order $q^m p^n$, and we have
\begin{align*}
G=&\langle a,b_1,b_2, \cdots,b_n|1=a^{q^m}=b_1 ^p=b_2 ^p= \cdots =b_n ^p,\,b_i b_j =b_j  b_i,\,i,j =1,2,\cdots,n,\\
& b_i ^a =b_{i+1},\,i,j =1,2,\cdots,n-1,b_n ^a=b_1 ^{d1} b_2 ^{d_2} \cdots b_n ^{d_n}  \rangle,
\end{align*}
where $f(x) =x^n-d_n x^{n-1}- \cdots -d_2 x-d_1$ is irreducible  on $F_p$, which is a factor of $ x^{q^m}-1$.
\item[(2)] $G = \langle a,x : a^{p^n} = b^p =1,\,x^{-1} a x = a^{1+p^{n-1}} \rangle$, $p$ an odd prime, $n \geq 2$.
\item[(3)] $G = \langle a,b,x: a^{p^n} = b^p = x^p =1,\,[x,a]=b,\,[a,b] = [b,x] =1 \rangle$ for some prime $p$.
\item[(4)]  $G =PQ$ is a minimal non-supersolvable group with $Q \unlhd G$, where $Q \in {\rm{Syl}}_q (G)$ is elementary abelian of order $|Q| > q >p$, and $P \in {\rm{Syl}}_p (G)$ is cyclic ,acting irreducibly on $Q$.
\item[(5)] $ G = \langle a, b, c : a^p = b^p = 1, c^{q^n}= 1, c^{-1}ac = a^r, [b, a] = [b, c] = 1,r \not\equiv 1 \,(\!\!\!\mod p), r^q \equiv 1 \,(\!\!\!\mod p) \rangle$, where $p >q$ are two primes.
\item[(6)] $G \cong {\rm SL}(2,3)$.
\end{itemize}
\end{theorem}
\begin{proof}
Since $G$ is not a PE-group, $G$ is not an NSN-group. By the definition of ON-group, one can easily find that any proper subgroup of $G$ is a PNC-group, and an NSN-group as well. It follows from {{\cite[Main Theorem]{NSN}}} that $G$ is solvable. It follows from {{\cite[Main Theorem]{NSN}}} that $G$ is isomorphic to one of 13 kinds of groups. By Theorem \ref{classification}, $G$ is isomorphic to one of six kinds of groups. Now we finish the proof by enumeration on the 13 kinds of groups in {{\cite[Main Theorem]{NSN}}}.

(1) It is (5) in Theorem \ref{classification}.

(2) Since this group is generated by three elements, it is (3) in Theorem \ref{classification} with $n=1$. 

(3) Since this $p$-group is generated by two elements, it is (2) in Theorem \ref{classification} with $n=1$.

(4) It follows from the proof of (5)(ii) in Theorem \ref{classification} that only $G = {\rm SL}(2,3)$ satisfies the condition.

(5) This is the generalized quaternion group. But $Q_{16}$ is a PE-group, failing to satisfy the condition.

(6) It is (6) in Theorem \ref{classification}.

(7) It is the group of order $p^i q^j$. If it is isomorphic to (4) in Theorem \ref{classification}, then it contradicts with the fact that $Q_8$ is not abelian. It follows that either it is isomorphic to (5),(6) in Theorem \ref{classification}, or it is not a non-PE-group, each of whose proper subgroup is an ON-group.

(8) It is the group of order $p^i q^j$. Since $G$ has non-cyclic Sylow $q$-subgroup, $q <p$, it follows that $G$ is not isomorphic to (4) in Theorem \ref{classification}.Therefore we conclude that either it is isomorphic to (5),(6) in Theorem \ref{classification}, or it is not a non-PE-group, each of whose proper subgroup is an ON-group.

(9) It is the group of order $p^i q^j$. Applying the same argument  used in (8), one can conclude that either it is isomorphic to (5),(6) in Theorem \ref{classification}, or it is not a non-PE-group, each of whose proper subgroup is an ON-group.

(10) It is the group of order $p^i q^j$. Since the order of the Sylow $p$-subgroup of $G$ is  equal to $p \neq p^2$ with $p >q$, 
 applying the same argument used in (8), one can indicate that either it is isomorphic to (5),(6) in Theorem \ref{classification}, or it is not a non-PE-group, each of whose proper subgroup is an ON-group.

(12)(13) It is the group of order $p^a q^b r^c$ with $|\pi (G)| \geq 3$, which is not isomorphic to (1)-(6) in Theorem \ref{classification}.

(11) We predicate that this group is a special case of (4) in Theorem \ref{classification}, and $G$ is a non-PE-group, each of whose proper subgroup is an ON-group. One can easily find that $\langle b \rangle \leq Z(G)$, hence we have $N_G (\langle ab \rangle) \geq \langle b \rangle$. As $(ab )^c =a^r b \in  \langle ab \rangle^G$, we indicate that $a^{r-1} \in \langle ab \rangle^G$. It follows from $r \not\equiv 1 \,(\!\!\!\mod p)$ that $a \in \langle ab \rangle^G$, therefore we have $b \in \langle ab \rangle^G$, i.e. $\langle b \rangle \leq N_G (\langle ab \rangle) \cap \langle ab \rangle^G$. Hence $\langle ab \rangle$ is not an NE-subgroup of $G$. On the other hand, we have
$$c^{-q} a c^q = c^{-q +1} a^r c^{q -1} = \cdots = a^{r^q}=a.$$
Thus we indicate that $\langle c^q \rangle \leq Z(G)$. Now let $U < G$. Since any element $g$ of $G$ can be uniquely represented as $g = a^i b^j c^k$, if $q|k$ holds for any $g \in U$, then we conclude that $U$ is abelian. Otherwise there exists $g = a^i b^j c^k \in U$ such that $q \nmid k$. We predicate that $o(a^i c^k)=q^n$, if $(k,q)=1$. As a matter of fact, it follows from $c^{-k} a^i c^k =c^{-k+1} a^{ri} c^{k-1}  = \cdots = a^{i r^k}$ that
$$a^i c^k a^i c^k = a^i c^{2k} a^{i r^k} = c^{2k} c^{-2k} a^i c^{2k} a^{i r^k} = c^{2k} a^{i r^{2k} +i r^k}. $$
Induction on $n$ gives that $(a^i c^k)^{m} = c^{mk} a^{i r^{mk} +i r^{(m-1)k} + \cdots +i r^k}$, thus we conclude that 
$$(a^i c^k)^{m} = c^{mk} a^{i r^{mk} +i r^{(m-1)k} + \cdots +i r^k} =c^{mk} a^{\frac{r^k(r^{mk} -1)}{r^k -1}i} =c^{mk} a^{\frac{r^k((r^{k})^m -1)}{r^k -1}i} .  $$
Let $m=q$, it follows that 
$$(a^i c^k)^{q} =  c^{qk} a^{\frac{r^k((r^{k})^q -1)}{r^k -1}i} .$$
If $r^k \equiv 1 \,(\!\!\!\mod p)$, then we get that $r^{(q,k)} \equiv 1 \,(\!\!\!\mod p)$. Since $(k,q)=1$, we indicate that $r \equiv 1 \,(\!\!\!\mod p)$, a contradiction. Hence $p \nmid r^k -1$, it follows that 
$$v_p \left( \frac{r^k((r^{k})^q -1)}{r^k -1}i \right) = v_p ((r^{k})^q -1) = v_p ((r^{q})^k -1) \geq v_p (r^{q} -1) \geq 1.$$
Thus $(a^i c^k)^{q}  = c^{qk}$, which implies that $o(a^i c^k) = q^n$. If $(j,p)=1$, then we indicate that $o(g) =o( a^i b^j c^k)= p q^n$, which forces $U = \langle  a^i b^j c^k \rangle$. Hence $U$ is a PNC-group. Otherwise $p|j$, i.e. $g = a^i c^k$ with $o(a^i c^k) = q^n$. If $U = \langle  a^i  c^k \rangle$, then $U$ is a PNC-group. If $U > \langle  a^i  c^k \rangle$, then we conclude that $|U| = q^n p$. Since $G$ is $p$-closed, we may assume that $U = \langle a^u b^v \rangle \langle  a^i  c^k \rangle$. If $p|u$, then $U$ is abelian. Otherwise $p \nmid u$, without loss of generality, $u=1$. It follows that $o(a b^v) = p$. Since $b^{-iv} a^{-i} a^i  c^k = c^k b^{-iv} \in U$, we indicate that $(c^k b^{-iv})^p = c^{kp}$, i.e. $\langle c \rangle \leq U$. Therefore we have $U =  \langle a b^v \rangle \langle  c \rangle$. As $(a b^v)^c = a^c b^v = a^r b^v \in U$, we get that $a^{r-1} \in U$. It follows from $p \nmid r-1$ that $\langle a \rangle \leq U$. Hence we conclude that $U = \langle a  \rangle \rtimes \langle  c \rangle$. Now we predicate that $U$ is a PNC-group. Let $H$ be a subgroup of $U$. If $H$ is a $p$-group, then clearly $H$ is an NE-subgroup of $U$. If $H$ is a $q$-group, without loss of generality we may assume that $H = \langle c^l \rangle \leq \langle c \rangle,\,l = q^t$. One can easily find that $N_U (H) \geq \langle c \rangle$. Induction on $n$ gives that 
$$a^{-1} c^n a = c^n a^{1-r^n} .$$
Hence we have $(c^l)^a = c^l  a^{1-r^{l}}= c^l a^{1-r^{q^t}}$. If $t=0$, then $H =\langle c \rangle$, an NC-group of $U$. If $t >0$, it follows that $p | 1-r^{q^t}$. Hence we conclude that $(c^l)^a = c^l$, which implies that $\langle a \rangle \leq N_U (H)$. Thus we indicate                                                      that $N_U (H) = U$, i.e. $H \unlhd U$. If $H$ is neither a $q$ group nor a $p$ group, without loss of generality, $H = \langle a \rangle \langle c^l \rangle,\,l = q^t$. It follows directly that $\langle c \rangle \leq N_U (H)$, therefore $N_U (H) = U$, i.e. $H$ is an NC-group of $U$. By the randomness of $H$, $U$ is a PNC-group. By the choice of $U$, we conclude that (11) is a special case of (4) in Theorem \ref{classification}.
\end{proof}
\begin{theorem}
Let $G$ be a finite group. Then $G$ is an ON-group if and only if $G$ is a Dedekind group or or $G$ satisfies the following conditions:
\begin{itemize}
\item[(1)] $G = P_1 \cdot P_2 \cdots P_{n} \cdot P$, where $P_1,P_2,\cdots,P_n$ are normal abelian Sylow $p_i$-subgroups of $G$, and $P = \langle x \rangle$ is a cyclic Sylow $p$-subgroup with $\langle x^p \rangle = O_p (G)$.
\item[(2)] $N_G (P) = P$.
\item[(3)] Let $H_1 :=P_1 \cdot P_2 \cdots P_{n}$. Then for any $1 \neq w \in H_1$, $w^x = w^m$, where $(m,o(w))=(m-1,o(w))=1$.
\end{itemize}
\end{theorem}
\begin{proof} The proof is proceeded via the two parts. 
\begin{itemize}
\item[(i)] Necessity of the proof.
\end{itemize}
Since $G$ is an ON-group, we conclude that $G$ is an NSN-group. By  {{\cite[Theorem 13]{MA}}}, $G$ is a Dedekind group or satisfies the following conditions:
\begin{itemize}
\item[(1')] $G$ contains an abelian normal subgroup $H$ of prime index $p$.
\item[(2')] a Sylow $p$-subgroup $P$ of $G$ is cyclic and self-normalizing in $G$.
\item[(3')] Let $P  =\langle x \rangle$. Then $w^x = w^m$ holds for any $w \in H$, where $m$ is an integer.
\end{itemize}
It is clear that $G = HP$. Let $H = P_1 \cdot P_2 \cdots P_{n} \cdot P^{*}$, where $P_i,i=1,2,\cdots,n$ are Sylow $p_i$-subgroups of $H$, and $P^{*}$ is a Sylow $p$-subgroup of $H$. It follows directly from (1'),(3') that $P^{*} = P \cap H = \langle x^p \rangle = O_p (G)$, and any subgroup of $H$ is normal in $G$. Therefore we indicate that $G= P_1 \cdot P_2 \cdots P_{n} \cdot P$. For any $w \in H$, since $w^x = w^m$, we indicate that $(m,o(w))=1$. Now let $H_1 :=P_1 \cdot P_2 \cdots P_{n}$. Suppose that there exists $w \in H_1$ such that $(m-1,o(w)) \neq 1$. It is clear that $m \not\equiv 1\,(\!\!\!\mod o(w)) $, thus we may assume that $p_i | (m-1,o(w))$. It follows from $x^{-1} w^{-1} x = w^{-m}$ that $x^w = x w^{-m+1}$. Since $N_G (P) = P$, by the definition of ON-group we have $P^G = G$. Let $w = g_1 g_2 \cdots g_n$, $g_j \in P_j$. It follows from $p_i |o(w)$ that $g_i \neq 1$. A trivial observation gives that $x^{w_1 w_2} = (x w_1 ^{-m_1 +1})^{w_2} = x w_1 ^{-m_1 +1} w_2 ^{-m_2 +1} $. Hence for $x^{g_j} = x g_j ^{-m_j +1}$, $j=1,2,\cdots,n$, we have 
$$x^w = x w^{-m+1} = x (g_1 \cdots g_n)^{-m+1} = x g_1 ^{-m_1 +1} g_2 ^{-m_2 +1} \cdots g_n ^{-m_n +1} . $$
Therefore we conclude that $0 \equiv m-1 \equiv m_i-1 \,(\!\!\!\mod p_i)$. Now let $P_i = C_{1,i} \times C_{2,i} \times \cdots \times C_{t_i,i},\,C_{j,i} = \langle h_{j,i} \rangle,\,j=1,2,\cdots,t_i$, $g_i = h_{1,i} ^{u_1}  h_{2,i} ^{u_2} \cdots h_{t_i,i} ^{u_i}$, and $h_{j,i}^x = h_{j,i} ^{m_{j,i}},\,j=1,2,\cdots,t_i$. It follows that 
$$x^{g_i} = x h_{1,i} ^{u_1 (-m_{1,i}+1)} h_{2,i} ^{u_2 (-m_{2,i}+1)} \cdots h_{t_i,i} ^{u_{t_i} (-m_{t_i,i}+1)} = x g_i ^{-m_i +1}  = x h_{1,i} ^{u_1 (-m_{i}+1)} h_{2,i} ^{u_2 (-m_{i}+1)} \cdots h_{t_i,i} ^{u_{t_i} (-m_{i}+1)} .$$
Without loss of generality we may assume that $v_{p_i} (u_1) < v_{p_i} (o(h_{1,i}))$, which implies that $h_{1,i} ^{u_1} \neq 1$.   Then we indicate that $0 \equiv u_1 (m_{1,i}-1) \equiv u_1 (m_i-1) \,(\!\!\!\mod o(h_{1,i}))$. Hence we have $p_i | m_{1,i}-1$. We predicate that $h_{1,i} \notin P^G$. For any $w = g_1 g_2 \cdots g_{i-1} h_{1,i} ^{u_1}  h_{2,i} ^{u_2} \cdots h_{t_i,i} ^{u_{t_i}} g_{i+1} \cdots g_n$, where $g_j \in P_j,\,u_j \in \mathbb{N}$, we conclude that
$$x^w  = x g_1 ^{-m_1 +1} g_2 ^{-m_2 +1} \cdots g_{i-1} ^{-m_{i-1} +1} h_{1,i} ^{u_1 (-m_{1,i}+1)}h_{2,i} ^{u_2 (-m_{2,i}+1)} \cdots h_{t_i,i} ^{u_{t_i} (-m_{t_i,i}+1)} g_{i+1} ^{-m_{i+1} +1} \cdots g_{n} ^{-m_{n} +1} .$$
Induction on $q \in \mathbb{N}$ gives that 
$$(x^{q})^{w} = x^q g_1 ^{(-m_1 ^{q} +1)} g_2 ^{(-m_2 ^{q} +1)} \cdots g_{i-1} ^{(-m_{i-1} ^{q} +1)} h_{1,i} ^{ u_1 (-m_{1,i}  ^{q}+1)}h_{2,i} ^{u_2 (-m_{2,i}  ^{q}+1)} \cdots h_{t_i,i} ^{ u_{t_i} (-m_{t_i,i}  ^{q}+1)} g_{i+1} ^{(-m_{i+1} ^{q} +1)} \cdots g_{n} ^{(-m_{n} ^{q} +1)} .$$
Also, for any $(x^{q_1})^{w_1},(x^{q_2})^{w_2} \in P^G$, it follows that 
$$(x^{q_1})^{w_1} (x^{q_2})^{w_2} = x^{q_1} w_1 ^{'} x^{q_2} w_2 ^{'} = x^{q_1 +q_2} \left( {w_1 ^{'}}\right)^{s_1 ^{q_2}}w_2 ^{'}, ~\text{where}~ \left( {w_1 ^{'}}\right)^{x} = \left( {w_1 ^{'}} \right)^{s_1}, s_1 \in \mathbb{N}.$$
We conclude from the two identities above that for any $x \in P^G$, the exponent of $h_{1,i}$ in the unique factorization induced by $G = P_1 \cdot P_2 \cdots P_{n} \cdot P$ is a multiple of $p_i$. Therefore we have $h_{1,i} \notin P^G$, a contradiction to the fact that $P^G = G$. Thus we indicate that $(o(w),n-1) =1$ holds for any $1 \neq w \in H_1$.
\begin{itemize}
\item[(ii)] Sufficiency of the proof.
\end{itemize}
If $G$ is a Dedekind group or satisfies (1)-(3), the last case suggests that $H:= H_1 \langle x^p \rangle$ is a normal abelian subgroup of $G$, and (1'),(2') holds. One can easily find from (3) and $O_p (G) = \langle x^p \rangle$ that every subgroup of $H$ is normal in $G$, hence (3') holds. Therefore $G$ is an NSN-group. Let $U$ be a subgroup of $G$. If $U \ntrianglelefteq G$, without loss of generality, $P \in {\rm{Syl}}_p (U)$. It suffices to prove that $P^G = G$. In fact, for any $1 \neq w \in H_1$, it follows from $w^x = w^m$ that $x^w = x w^{-m+1}$. By $(m-1,o(w))=1$, we conclude that $H \leq P^G$. Thus $G = HP \leq P^G$ and we indicate that $G$ is a PNC-group.
\end{proof}

\begin{theorem}
Let $G$ be a group such that any maximal subgroup of $G$ is a solvable PNC-group, then one of the statements is true:
\begin{itemize}
\item[(1)] $G$ is 2-nilpotent.
\item[(2)] $G$ is not 2-nilpotent, and $G$ is a minimal non-abelian group of order $q^m 2^n$, and we have
\begin{align*}
G=&\langle a,b_1,b_2, \cdots,b_n|1=a^{q^m}=b_1 ^2=b_2 ^2= \cdots =b_n ^2,\,b_i b_j =b_j  b_i,\,i,j =1,2,\cdots,n,\\
& b_i ^a =b_{i+1},\,i,j =1,2,\cdots,n-1,b_n ^a=b_1 ^{d1} b_2 ^{d_2} \cdots b_n ^{d_n}  \rangle,
\end{align*}
where $f(x) =x^n-d_n x^{n-1}- \cdots -d_2 x-d_1$ is irreducible  on $F_2$, which is a factor of $ x^{q^m}-1$.
\end{itemize}
\end{theorem}
\begin{proof}
Since any maximal subgroup of $G$ is a solvable PNC-group, it follows from {{\cite[Theorem 3.2]{LG}}} that either $G$ is 2-nilpotent, or $G$ is a minimal non-nilpotent group. The second case suggests that one of the following results holds:
\begin{itemize}
\item[(i)] $G = P \rtimes Q$, $P$ is an elementary abelian 2-group, $Q$ is a cyclic $q$-group, where $q \neq 2$. 
\item[(ii)] $G = P \rtimes Q$, $P$ is quaternion group, $Q$ is a cyclic 3-group.
\end{itemize}
In the two cases, $\pi (G) =2$. By Lemma \ref{nilpotent}, we conclude that every proper subgroup of $G$ is a Dedekind group. It follows that $G$ is isomorphic to (1)-(3) in {{\cite[Theorem]{SC}}}. Since (1),(3) in {{\cite[Theorem]{SC}}} are 2-nilpotent, we indicate that $G$ is isomorphic to (2) in {{\cite[Theorem]{SC}}}. As (2) in {{\cite[Theorem]{SC}}} has cyclic Sylow subgroups, it follows that (ii) is not isomorphic to (2) in {{\cite[Theorem]{SC}}}. Hence (i) is isomorphic to (2) in {{\cite[Theorem]{SC}}}, and we have $p=2$.
\end{proof}
\begin{theorem}
Let $G$ be a non-abelian simple group. Then any second maximal subgroup  of $G$ is a solvable PNC-group if and only if $G$ is isomorphic to one of the following simple groups:
\begin{itemize}
\item[(1)] ${\rm{PSL}} (2,p)$, where $p$ is a prime such that $p >3$, $p^2 -1 \not\equiv 0 \,(\!\!\!\mod 5)$, $p^2 -1 \not\equiv 0 \,(\!\!\!\mod 16)$. 
\item[(2)] ${\rm{PSL}} (2,2^q)$, where $q$ is a prime such that $2^q -1$ is a prime.
\item[(3)] ${\rm{PSL}} (2,3^q)$, where $q$ is an odd prime such that $\frac{3^q -1}{2}$ is a prime.
\end{itemize}
\end{theorem}
\begin{proof} The proof is proceed via the two parts:
\begin{itemize}
\item[(1)] Sufficiency of the proof.
\end{itemize}
If $G$ is isomorphic to one of the three groups above, it follows from {{\cite[Lemma 5.2]{YL}}} that the maximal subgroup of $G$ is isomorphic to one of the following three types:
\begin{itemize}
\item[(i)] Dihedral group $D_{\frac{r-1}{d}}$ of order $\frac{2(r-1)}{d}$, where $d = (r-1,2)$, $r=p,2^q,3^q$.
\item[(ii)] A Frobenius group $N$ with elementary abelian Frobenius kernel $K$ of order $r$, and cyclic Frobenius complement $H$, where $d = (r-1,2)$, $r=p,2^q,3^q$.
\item[(iii)] $A_4$.
\end{itemize}
Since $p^2 -1 \not\equiv 0 \,(\!\!\!\mod 16)$, we conclude that either $4 \parallel p-1$ or $4 \parallel p+1$. The second case implies that $2 \parallel p-1$. If $r=p$, one can easily find that $\frac{r-1}{d}$ is odd or $2 \parallel \frac{r-1}{d}$. By proposition \ref{dihedral group}, the maximal subgroups of $D_{\frac{r-1}{d}}$ are exactly $D_{\frac{r-1}{dp}}$, $p \in \pi (|\frac{r-1}{d}|)$ and $C_{\frac{r-1}{d}}$. It follows that 
$$v_2 \left(  \frac{r-1}{dp}  \right)  \leq v_2 \left( \frac{r-1}{d} \right) \leq 1,\,\forall \, p \in \pi (|\frac{r-1}{d}|).$$
Again by proposition \ref{dihedral group}, we indicate that any subgroup of $D_{\frac{r-1}{d}}$ is a solvable PNC-group, if $r=p$. For $r = 2^q$, we conclude from $\frac{2^q -1}{d} = 2^q -1$ is odd that any maximal subgroup of $D_{\frac{r-1}{d}} = D_{2^q -1}$ is a solvable PNC-group. For $r= 3^q$, since $\frac{3^q -1}{2}$ is a prime, we have 
$$v_2 \left( \frac{r-1}{d}  \right) = v_2 \left( \frac{3^q-1}{2}  \right) \leq 1.$$
Hence any maximal subgroup of $D_{\frac{r-1}{d}} = D_{\frac{3^q -1}{2}}$ is a solvable PNC-group. For (ii), it follows from {{\cite[Lemma 5.2]{YL}}} that $N$ is a minimal non-abelian group. Hence any subgroup of $N$ is abelian, and any maximal subgroup of $N$ is a solvable PNC-group. One can easily find that any maximal subgroup of $A_4$ is a solvable PNC-group.
\begin{itemize}
\item[(2)] Necessity of the proof.
\end{itemize}
If any second maximal subgroup of $G$ is a solvable PNC-group, by proposition \ref{NE}, we conclude that any second maximal subgroup $H$ of $G$ is a PE-subgroup. It follows from {{\cite[Theorem 5.3]{YL}}} that $G$ is isomorphic to one of the following groups:
\begin{itemize}
\item[(1)] ${\rm{PSL}} (2,p)$, where $p$ is a prime such that $p >3$, $p^2 -1 \not\equiv 0 \,(\!\!\!\mod 5)$, $p^2 -1 \not\equiv 0 \,(\!\!\!\mod 16)$. 
\item[(2)] ${\rm{PSL}} (2,2^q)$, where $q$ is a prime such that $2^q -1$ is a prime.
\item[(3)] ${\rm{PSL}} (2,3^q)$, where $q$ is an odd prime such that $\frac{3^q -1}{2}$ is a prime.\qedhere
\end{itemize}
\end{proof}
\begin{theorem}
Let $G$ be a non-solvable group which is not simple. If any second maximal subgroup of $G$ is a solvable PNC-group, then $G$ is isomorphic to one of the following groups:
\begin{itemize}
\item[(1)] ${\rm{SL}} (2,p)$, where $p$ is a prime such that $p >3$, $p^2 -1 \not\equiv 0 \,(\!\!\!\mod 5)$, $p^2 -1 \not\equiv 0 \,(\!\!\!\mod 16)$. 
\item[(2)] ${\rm{SL}} (2,3^r)$, where $r$ is an odd prime.
\end{itemize}
\end{theorem}
\begin{proof}
It follows from proposition \ref{NE} (10) that any second maximal subgroup of $G$ is a PE-group. Since project special linear groups are simple, it follows from {{\cite[Theorem 3.3]{LG}}} that $G$ is isomorphic to the following two kinds of special linear groups:
\begin{itemize}
\item[(1)] ${\rm{SL}} (2,p)$, where $p$ is a prime such that $p >3$, $p^2 -1 \not\equiv 0 \,(\!\!\!\mod 5)$, $p \equiv 3,5\,(\!\!\!\mod 8)$. 
\item[(2)] ${\rm{SL}} (2,3^r)$, where $r$ is an odd prime.\qedhere
\end{itemize}
\end{proof}

\!\!\!\!\!\!\!\!\!\textbf{Declarations of interest:} none.

\end{document}